\documentclass[12pt]{article}
\usepackage{amsmath, amssymb, amsthm}
\usepackage{graphicx}
\usepackage{enumerate}
\usepackage{natbib}
\usepackage{url} 
\RequirePackage[colorlinks,citecolor=blue,urlcolor=blue]{hyperref}
\usepackage{algorithm}

\usepackage[noend]{algpseudocode}

\usepackage{xcolor}
\usepackage{tikz} 

\usepackage{ulem}
\newtheorem{theorem}{Theorem}
\newtheorem{coro}[theorem]{Corollary}
\newtheorem{prop}[theorem]{Proposition}

\newtheorem{assumption}{Assumption}
\newtheorem{remark}{Remark}
\newtheorem{example}{Example}
\newtheorem{lemma}[theorem]{Lemma}

\newcommand{\blind}{0}

\newcommand{\arxiv}{1}

\def\g{\boldsymbol}

\newcommand{\einsfun}{\g 1} 

\newcommand{\KL}{Kullback--Leibler}

\newcommand{\norm}[1]{\left\Vert#1\right\Vert}
\newcommand{\abs}[1]{\left\vert#1\right\vert}
\newcommand{\set}[1]{\left\{#1\right\}}
\newcommand{\suit}[1]{\left(#1\right)}

\newcommand{\Var}{\mathbb V}
\newcommand{\Cov}{\mbox{Cov}}

\newcommand{\td}{\stackrel{d}{\to}}
\newcommand{\tp}{\stackrel{P}{\to}}

\begin{document}

\def\spacingset#1{\renewcommand{\baselinestretch}%
{#1}\small\normalsize} \spacingset{1}


\if0\blind
{
  \title{\bf 
  A Kullback--Leibler divergence test for multivariate extremes: theory and practice}   
  \author{Sebastian Engelke\thanks{
    S.~Engelke acknowledges support from the Swiss National
Science Foundation under Grant 186858.}\hspace{.2cm}\\
    Research Institute for Statistics and Information Science,\\ University of Geneva, Switzerland\\
    and \\
    Philippe Naveau \thanks{
Part of P.~Naveau’s research work was supported by the   Agence Nationale de la Recherche via:  the SHARE PEPR Maths-Vives project (France 2030 ANR-24-EXMA-0008), EXSTA, PORC-EPIC, the PEPR TRACCS program  (PC4 EXTENDING, ANR-22-EXTR-0005), and  the PEPR   IRIMONT (France 2030 ANR-22-EXIR-0003). He has also benefited  from the Geolearning research chair.} \\
    Laboratoire des Sciences du Climat et de l’Environnement (CNRS), France\\
    and \\
    Chen Zhou\\
    Econometric Instistute, Erasmus School of Economics,\\
    Erasmus University Rotterdam, The Netherlands
    }
  \maketitle
} \fi

\if1\blind
{
  \bigskip
  \bigskip
  \bigskip
  \begin{center}
    {\LARGE\bf  
  A Kullback--Leibler divergence test for multivariate extremes: theory and practice\\
}
\end{center}
  \medskip
} \fi

\bigskip
\begin{abstract}

Testing whether two multivariate samples exhibit the same extremal behavior is an important problem in various fields including environmental and climate sciences. While several ad-hoc approaches exist in the literature, they often lack theoretical justification and statistical guarantees. On the other hand, extreme value theory provides the theoretical foundation for constructing asymptotically justified tests. We combine this theory with Kullback--Leibler divergence, a fundamental concept in information theory and statistics, to propose a test for equality of extremal dependence structures in practically relevant directions. Under suitable assumptions, we derive the limiting distributions of the proposed statistic under null and alternative hypotheses. Importantly, our test is fast to compute and easy to interpret by practitioners, making it attractive in applications. Simulations provide evidence of the power of our test. In a case study, we apply our method to show the strong impact of seasons on the strength of dependence between different aggregation periods (daily versus hourly) of heavy rainfall in France.

\end{abstract}

\noindent%
{\it Keywords:} 
Two-sample statistical test; multivariate regular variation; seasonal precipitation extremes; intensity-duration-frequency statistic
\vfill

\newpage
\spacingset{1.8} 

\section{Introduction}\label{sec: intro}

In environmental risk analysis, a recurrent question is to determine if two multivariate extreme events are statistically indistinguishable or not.
This inquiry is, for example, at the core of the field of extreme event attribution in climate science \cite[e.g.][]{naveau_statistical_2020} whose aim is to contrast extreme event probabilities  in a world with and without anthropogenic forcings. 
To identify  significant changes in compound climate extremes,  different statistical approaches have been used by the climate community \citep[e.g.][]{Zscheischler20review}, but they lack statistical guarantees.  
Besides climate studies, other examples motivating the need of metrics to contrast multivariate extremal behaviors can be found in 
systemic risk analysis in finance \citep[e.g.][]{emb1997}  and environmental sciences \citep[e.g.][]{hus2013,bui2008}.

Flood risk managers, for instance, are interested in determining if, beyond the marginal differences among  different rainfall gauges, extremal dependencies vary significantly across seasons, in space or according to temporal aggregation scales.    
To illustrate this point, we compare, for the city of Bordeaux in France over the period 2006--2023,  daily maxima of precipitation at the 6-minute scale with the  hourly scale; see the left panels of Figure \ref{fig:bordeaux} where the scatterplots in top and bottom rows correspond to the winter and fall seasons, respectively. 
Comparing the seasonal effect on rainfall  aggregations  at different time scales (here 6-min versus hourly) represents  a key research topic in hydrology that goes back at least half century; see, e.g., the intensity-duration-frequency curves 
of the  U.S.~Weather Bureau  \citep{Hershfield61}.
There is a large body of literature  on  modeling and interpreting the marginal rainfall distributions at the different aggregation scales \cite[e.g.,][]{Ulrich20,haruna2023modeling}. It is  natural to wonder if the joint extremal behaviors also change with seasons, i.e., do  the  gray and black clouds in Figure \ref{fig:bordeaux}  differ from each other?

\begin{figure}[tb]
\centering
    \includegraphics[trim={0cm 5cm 0 5cm},clip, width=.8\linewidth]{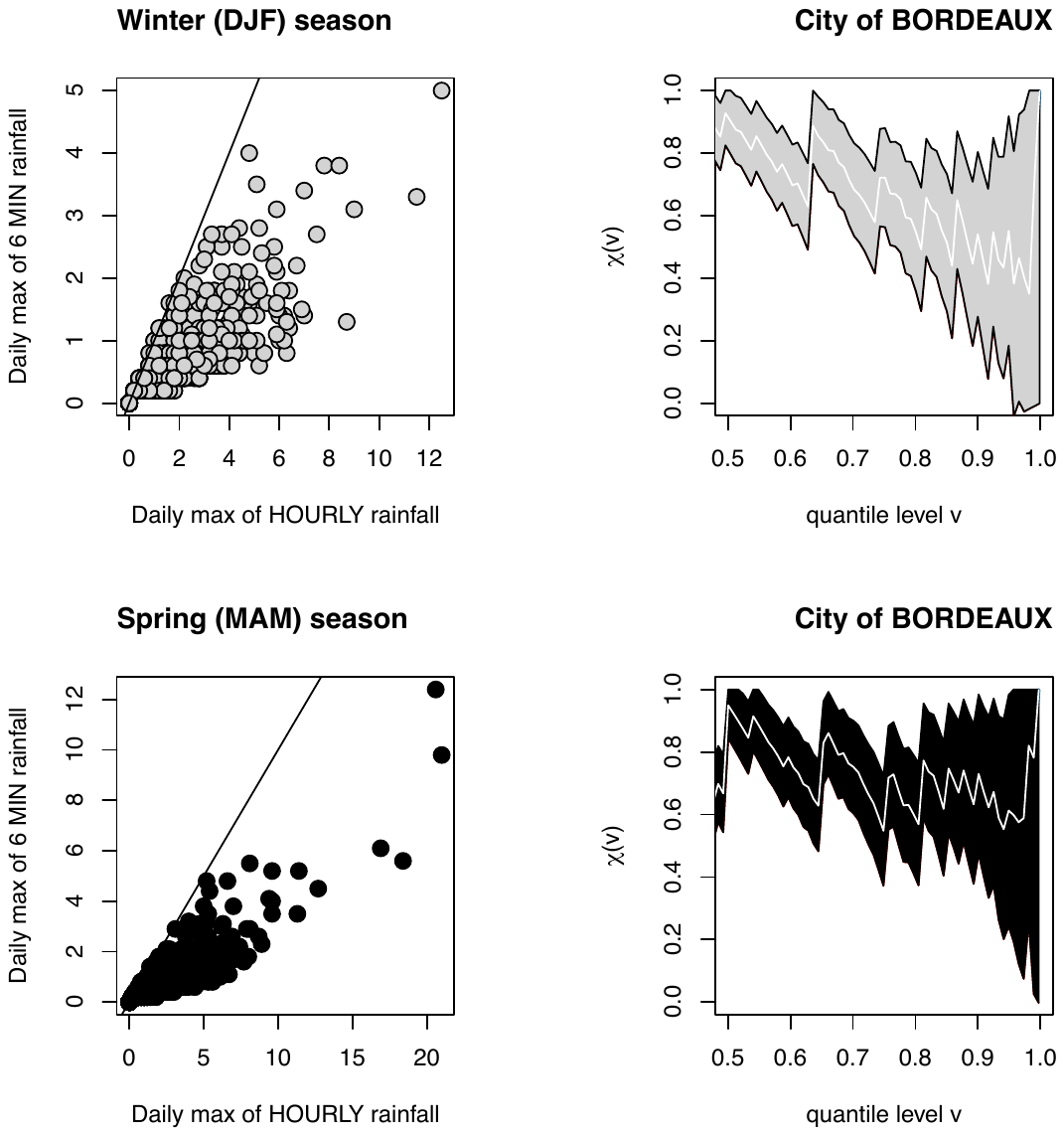}
	\caption{Left: scatter plots of daily maxima of hourly and 6-minute rainfall recorded in Bordeaux (France) in winter (top) and spring (bottom) during the period 2006--2023. Right: extremal correlation coefficient $\chi(v)$ with confidence intervals for this data, measuring the strength of dependence above the quantile level $v$; see also \eqref{eq: chi}. 
    }
	\label{fig:bordeaux}
\end{figure}

Multivariate extreme value theory provides the mathematical foundation to describe extremal dependence structures and to address such questions \cite[e.g.,][]{beirlant_statistics_2005,deh2006a}.
The first, somewhat simplistic attempt would be 
to contrast two bivariate extremal populations
by a summary statistic of the extremal dependence structure such as the extremal correlation $\chi$ \citep{col1999} defined in~\eqref{eq: chi} below. 
For our Bordeaux rainfall example, the right panels of Figure~\ref{fig:bordeaux} display the estimated extremal correlation (for different threshold values)
and their 95\% confidence intervals.
Visually and theoretically it is however difficult to conclude
whether winter and spring extremes have different extremal behaviors. Statistically speaking, the power of an associated test would be quite low. 
Beyond summary statistics, one could model the full extremal dependence either parametrically or non-parametrically \citep{bucher2017}.   
While this approach can improve the power, it demands significant expertise of multivariate extreme value theory and is computationally much more expensive.
In this work, we opt for a middle road to keep the computational simplicity and interpretability similar to a $\chi$ test, while adding flexibility in order to improve the power.

Mathematically, our goal is to test whether two $d$-dimensional random vectors $\g{\tilde X} = (\tilde X_1, \dots,  \tilde X_d)$ and $\g{\tilde Y} = (\tilde Y_1, \dots,  \tilde Y_d)$ have the same extremal dependence structure. Denote the $j$th marginal distributions of $\g{\tilde X}$ and $\g{\tilde Y}$ as $F_j$ and $G_j$, respectively. To avoid the impact of different marginal distributions (as for instance the 6-min and hourly rainfall distributions), we consider vectors $\g X$ and $\g Y$ with entries
\begin{align}\label{X_true}
  X_{j} = 1 / (1 - F_j(\tilde X_{j})), \qquad   Y_{j} = 1 / (1 - G_j(\tilde Y_{j})), \quad j = 1,\dots, d,
\end{align}
standardized to Pareto margins. Since this transformation 
does not change the copula, comparing the extremal dependence structure between $\g{\tilde X}$ and $\g{\tilde Y}$ can be achieved by comparing that between $\g{X}$ and $\g{Y}$.

For multivariate data, the definition of an extreme event is ambiguous. A common practice in extreme value theory is to summarize the vector of interest by a univariate risk functional $r: \mathbb [0,\infty)^d \to [0,\infty)$, which characterizes the types of events that are particularly relevant or impactful in specific application of interest.
An extreme observation of, say, $\g X$ is then defined by the $r$-exceedance set $\{r(\g X) > u\}$, where the high threshold $u>0$ specifies the magnitude of the extreme event \citep[e.g.,][]{dom2015, Fondeville22}. 
Typical examples for risk functionals are $r(\g x) = \max(x_1,\dots, x_d)$, $r(\g x) = \min(x_1,\dots, x_d)$  or $r(\g x) = \sqrt{x_1^2 + \dots + x_d^2}$; see the three panels in Figure~\ref{fig:risks} in $d=2$.

\begin{figure}[tb]
\centering
\includegraphics[trim={0cm 11cm 0 11cm},clip,width=1\linewidth]{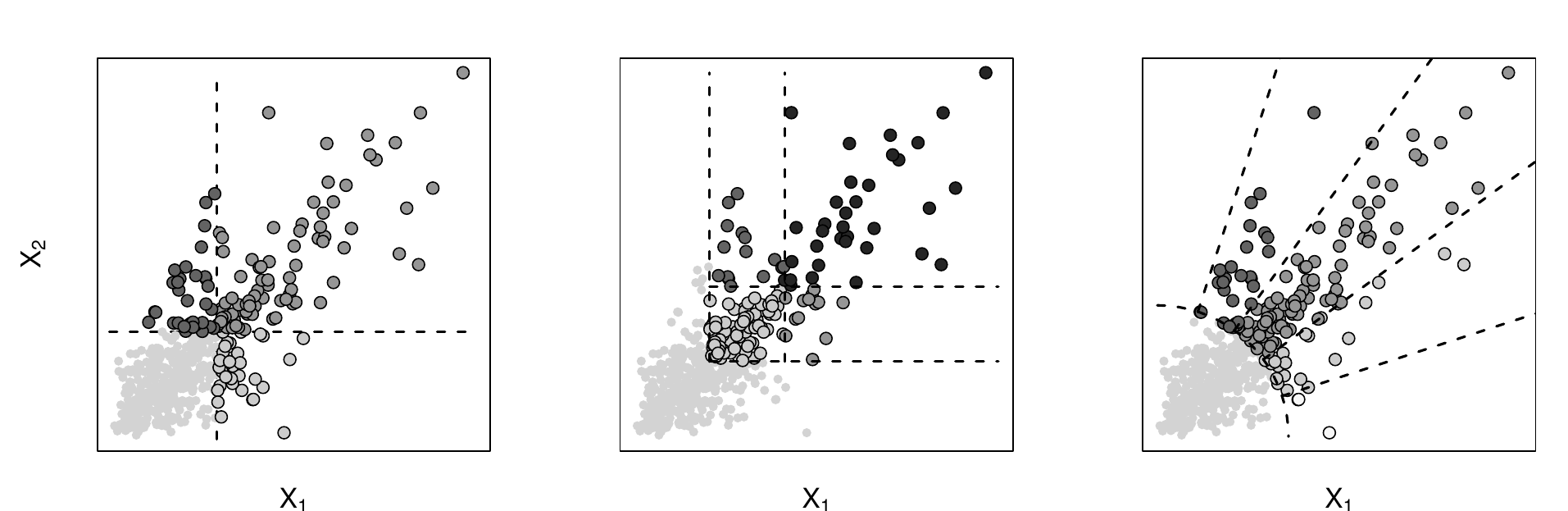}
	\caption{Three risk functionals $r(\g X)$ with  partition  examples obtained  on  
    simulated bivariate draws of $\g X_1,\dots, \g X_n$.
    Light gray dots are non extreme realizations and circled points in each panel represents extremes, i.e., $\{ \g X_i: r(\g X_i) >u \}$, defined with respect to the  risk functional  $r(\g x) = \max(x_1,x_2)$ (left panel); $r(\g x) = \min(x_1,x_2)$ (center panel); $r(\g x) = \sqrt{x_1^2 + x_2^2}$ (right panel). 
    The region  $\{ r(\g X) >u \}$ is divided in $K$  subsets in order to estimate 
    \eqref{cond_prob} with $K=3, 4$ and $5$ for the right, middle and right panel, respectively.
    }
	\label{fig:risks}
\end{figure}

On an intuitive level, our test consist of partitioning the exceedance set $\{ r(\g X) >u \}$ (the darker points in Figure~\ref{fig:risks}) in $K$ different different regions and check whether the number of extreme events in these sets differ between the samples of $\g X$ and $\g Y$; 
see the different shades of gray points in Figure~\ref{fig:risks} 
for examples of such partitionings. More mathematically, if $r$ is homogeneous of degree one, we may write 
$$
\{ r(\g X) >u \} =  u \Omega_r \mbox{ with } \Omega_r := \{\g x\in[0,\infty)^d:r(\g x)>1\}.
$$
For a  disjoint partition of Borel sets $A_1,\dots, A_K$  of  $\Omega_r$, our main assumption is the existence of the limit of conditional probabilities
\begin{equation}\label{cond_prob}
p_j := \lim_{u\to\infty}p_{j}(u):=\lim_{u\to\infty} P(\g X\in uA_j\mid r(\g X)>u) \in (0,1),\qquad j=1,\dots, K.
\end{equation}
This assumption is satisfied for many multivariate distributions, and it particular holds under the commonly used framework of multivariate regular variation~\citep{res2008}.
In the bivariate  case, a specific partition of $\Omega_r$ with $r(\g x)=\max(x_1,0)$ gives a direct link to the extremal correlation
\begin{equation}\label{eq: chi}
    \chi(v):= p_1(u) = P(X_2> u \mid X_1>u)= P(\g X\in uA_1\mid r(\g X)>u)
\end{equation}
where $v=1-1/u \in (0,1)$ and $A_1=[1,\infty)\times [1,\infty)$. 
This highlights that our approach extends the information given by $\chi$ by 
decomposing the exceedance set $\{ r(\g X) >u \}$ into smaller sub regions in order to capture more detailed information within this set. 

Analogously, we define probabilities $q_j$ for the random vector $\g Y$.  The main goal of this work is to test 
the null hypothesis
\begin{equation}\label{H0}
  H_0: p_j = q_j,\quad \mbox{ for all } j = 1,\dots, K,
\end{equation}
against the alternative that for at least one of these probabilities are not equal. If the extremal dependence structure of $\g X$ and $\g Y$ coincide then $H_0$ holds, or equivalently, if $H_0$ is rejected then the extremal dependence structures differ. Compared to testing the full extremal dependence structure as in \cite{bucher2017}, the simpler null hypothesis \eqref{H0} is 
often relevant in practice due to its straightforward interpretation. Indeed, focusing on particular sets allows us to define a fast-to-compute, easy-to-interpret and asymptotically well behaved test statistic.

Our test relies on the multinomial \KL{} (KL) divergence 
\begin{equation}\label{eq:KL_Multinomial}
{D}_K := \sum_{j=1}^K ( p_j-  q_j)(\log p_j-\log  q_j),
\end{equation}
between the two populations. The divergence $D_K$ is non-negative, and zero only if $H_0$ holds.
We construct an empirical estimator $\hat D_K$ of the theoretical $D_K$ based on samples of $\g X$ and $\g Y$, which is used to obtain a statistical test with a desired level $\alpha \in (0,1)$.
The statistic $\hat D_K$ is small if the probabilities of $\g X$ and $\g Y$ falling in the exceedance sets are similar, and when it is large we reject $H_0$. To formalize this, we derive an asymptotic large sample theory for $\hat D_K$. In particular, we show that if the marginal distribution functions in~\eqref{X_true} are known, then the properly normalized $\hat D_K$ behaves asymptotically like a $\chi^2(K-1)$ distribution. On the other hand, if marginal distributions are unknown we use empirical distribution functions for the standardization. The limit in this case is more complicated and we propose a bootstrap procedure to obtain critical values.

The advantage of our test for extremal dependence based on KL divergence is its ubiquity in applied sciences.
For example, \cite{nav2014} leveraged the KL divergence to study univariate climate extremes. Concerning the multivariate setup,
several studies have already applied our test statistic ${D}_K$ in~\eqref{eq:KL_Multinomial}, but with focus entirely on   applications rather than statistical properties of the test. For instance, \cite{Zscheischler20}  assessed whether  compound wind-precipitation extremes exhibit similar dependence structures in different data products. 
In a similar vein, our test was used for model evaluation, comparing the extremal dependence of observational data and climate models between oceanographic, fluvial and pluvial flooding drivers along the coastlines of the United States \citep{nas2021} and China \citep{li2024}. \cite{nas2023} tested for temporal changes in dependence between compound coastal and inland flooding.
A different perspective on the KL divergence $D_K$ was by interpreting it as a general distance measure for extremal dependence. Following this approach, \cite{vig2021} used it to obtain homogeneous clusters in terms of wind-precipitation extremes in Great Britain and Ireland.

Section~\ref{sec:test} introduces the KL divergence test and derives 
its asymptotic theory, with the cases of known and unknown 
marginals covered in Sections \ref{sec:known marginals} and \ref{sec:unknown marginals}, respectively.
To assess the performance of our test, a simulation study is detailed in Section \ref{sec: sim}.
Following the case study in Figure \ref{fig:bordeaux},
in Section~\ref{sec: Application} we
determine whether the seasonal cycle has a significant impact on the dependence structure between 6-min and hourly extreme precipitation; the data of this study are openly available at \url{https://meteo.data.gouv.fr/datasets/donnees-climatologiques-de-base-6-minutes/}.
All proofs can be found in the Appendix\if1\arxiv
{.
} \fi
\if0\arxiv
{
and the R code to reproduce all figures is included with this submission.
} \fi

\section{A two-sample KL divergence test for extremal dependence} \label{sec:test}
We work with $d$-dimensional random vectors $\g {\tilde X} = (\tilde X_1, \dots, \tilde X_d)$ and $\g {\tilde Y} = (\tilde Y_1, \dots, \tilde Y_d)$ and aim to test whether they share the same extremal dependence structure. We assume throughout that the standardized vectors $\g X$ and $\g Y$ defined in \eqref{X_true} are multivariate regularly varying, implying that the probabilities $p_j$ and $q_j$ are well defined.

We present basic background on multivariate regular variation in Section \ref{sec:background} and define our test statistic based on an observed sample in Section \ref{sec:method}. In practice, such a sample can be obtained if the marginal distributions of $\g{\tilde X}$ and $\g{\tilde Y}$ are known. For this case, we present statistical theories leading to size and power of the test in Section \ref{sec:known marginals}. Section \ref{sec:unknown marginals} handles the case where the marginal distributions are unknown. We show that the test is still valid, however, the critical value needs to be obtained via a bootstrap procedure.

\subsection{Background and multivariate regular variation}\label{sec:background}
We consider a generic random vector $\g X = (X_1, \dots , X_d)$ whose margins have already been standardized to Pareto margins according to~\eqref{X_true}. To model the extremal  dependence of $\g X$, a common approach is to assume  multivariate regularly varying (MRV) \citep{res2008}. This assumption is equivalent to the existence of a probability measure $\nu$, which we call the exceedance distribution, defined by
\begin{align}\label{mevd}
  \lim_{u\to \infty}  P(\g X/u \in  B | \norm{\g X}_\infty>u) = \nu(B),
\end{align}
for all Borel sets $B\subset \{ \g x \geq 0: \| \g x \|_\infty > 1\}$ such that 
$\nu(\partial B)=0$. It can be shown that $\nu$ satisfies the homogeneity property $\nu(tB)=t^{-1}\nu(B)$ for any $t>1$.

The measure $\nu$ describes the tail of the random vector $\g X$. Indeed, for a fixed $u>0$, the left-hand side of~\eqref{mevd} is the distribution of normalized exceedances, that is, the observations where at least one component is extreme; see the darker points in Figure~\ref{fig:risks}. Since each component $X_j$, $j=1,\dots, d$, follows the same marginal Pareto distribution, $\nu$ characterizes the dependence structure between the exceedances. Sending the threshold $u$ in~\eqref{mevd} to infinity guarantees that only the most extreme observations are considered, which is important to show proper theory for our testing procedure later.

\begin{remark}
  While each $X_j$ obtained from transformation~\eqref{X_true} has standard Pareto margins, the marginal distributions $F_j$ of the original data $\tilde X_j$ can be arbitrary. For instance, some components can be heavy-tailed, whereas other components  can have light tails or even finite upper endpoints. In climate science, this is important since interest is often in compound extremes with variables with fairly different tail behavior, such as joint temperature-precipitation \citep{zsc2017} or wind-precipitation  extremes \citep{Zscheischler20}.   
\end{remark}

While in~\eqref{mevd} an exceedance is an observation where at least one component is extreme, for multivariate data, the definition of an extreme event depends in general on the context of the domain of application.
This assumption characterizes also the exceedance distributions for other types of events. 
In order to have a general framework, as mentioned in the introduction, we use the notion of a homogeneous risk functional $r: [0,\infty)^d \to [0,\infty)$ satisfying $r( t \g x ) =t  \, r(\g x)$ for any $t>0$ and $x \in [0,+\infty)^d$. In addition, we assume that $r$ is measurable so that the set $\Omega_r := \{\g x\in[0,\infty)^d:r(\g x)>1\}$ is a Borel set bounded away from the origin. 
The next examples provide some typical risk functionals together with partitions of the corresponding sets $\Omega_r$.

\begin{example} \label{example: max}
  Consider the risk functional $r(\g x)=\max(x_1, \dots, x_d)$. For any non-empty $I\subset \{1,\dots, d\}$ define the set $A_I=\{x\in\mathbb R^d: x_j>1, j\in I; x_j\leq 1, j\in I^c\}$. Then $(A_I)_I$ is a partition of $K=2^d-1$ sets. The left-hand side of Figure~\ref{fig:risks} shows this partition in dimension $d=2$.
\end{example}

\begin{example} \label{example: min}
  Consider the risk functional $r(\g x)=\min(x_1, \dots, x_d)$. For any $I\subset \{1,\dots, d\}$ define the set $A_I=\{x\in\mathbb R^d: x_j>2, j\in I; 1<x_j\leq 2, j\in I^c\}$. Then $(A_I)_I$ is a partition of $K = 2^d$ sets. The center of Figure~\ref{fig:risks} shows this partition
    in dimension $d=2$.
\end{example}

\begin{example} \label{example: radius}
  For a general risk functional $r$ as above, let  $S_r = \set{x\in\mathbb R^d:r(\g x)=1}$ be the corresponding ``unit sphere''.
  For any partition $\tilde A_1, \dots, \tilde A_K$ of $S_r$, the sets $A_j= \{x \in\mathbb R^d: x/r(\g x) \in \tilde A_j, r(\g x) > 1\}$,  $j=1,\dots, K$, form a partition of $\Omega_r$. As an example, the right-hand side of Figure~\ref{fig:risks} shows this partition for the risk functional $r(\g x) = \sqrt{x_1^2 + x_2^2}$ in dimension $d=2$ with $K=5$.
\end{example}

The next result shows that under the MRV assumption, for a homogeneous risk functional $r$, the limiting probabilities in~\eqref{cond_prob} exist and have a simple representation in terms of an extended definition of~$\nu$: for any set $B \subset [0,\infty)^d\setminus \{0\}$ bounded away from the origin, that is, $\mu_B := \min\{\|\g x\|_\infty: \g x \in B\} > 0$, define $\nu(B) = \nu(B/\mu_B) / \mu_B.$ Note that this extended $\nu$, also called exponent measure \cite[Chapter 6]{deh2006a}, satisfies homogeneity with degree $-1$, for all $t >0$.

\begin{prop} \label{prop: values of p}
   Assume that the MRV assumption in \eqref{mevd} holds for $\g X$ with exponent measure $\nu$. Then for any Borel set $B\subset \Omega_r$ with $\nu(\partial B)=0$ 
    $$\lim_{u\to\infty} P(\g X\in uB|r(\g X)>u)=\frac{\nu(B)}{\nu(\Omega_r)}.$$
\end{prop}
The limiting probabilities $p_j$ and $q_j$, $j=1,\dots, K$, in~\eqref{cond_prob}
can be expressed in a similar way with respect to their exceedance distributions $\nu_X$ and $\nu_Y$, respectively.

\subsection{The \KL{} test statistic}\label{sec:method}
Let $\g X_1,\dots, \g X_n$ and $\g Y_1,\dots, \g Y_n$ be independent samples of the random vectors $\g X$ and $\g Y$, respectively.
Assume that both random vectors satisfy the MRV assumption~\eqref{mevd}.
Suppose the risk functional $r$ and the partitioning sets $A_1,\dots ,A_K$ are fixed. In the following, we detail an the algorithm to obtain an estimate of the \KL{} divergence test statistic for the null hypothesis~\eqref{H0} to assess the difference in extremal dependence between the two samples.   
In order to estimate the limiting probabilities in~\eqref{cond_prob}, we choose a data dependent threshold $u_n$ and estimate empirically the conditional probabilities
$p_j(u_n)$. To this end, denote by $R^{X}_{1,n}\leq\ldots\leq R^{X}_{n,n}$ the order statistics of the set $\set{r(\g X_i)}_{i=1}^n$. For some $k_n \leq n$ we set $u_n^X = R^{X}_{n-k_n,n}$, resulting in exactly $k_n$ exceedances since
\[ k_n = \sum_{i=1}^n\einsfun\{r(\g X_i)>R^{X}_{n-k_n,n}\}.\]
The empirical estimator of $p_j$, $j=1,\dots, K$, takes the form
\begin{align}
  \label{p_estimate}
  \hat p_{j} = \hat p_{j}(u_n^X) 
=\frac{1}{k_n}\sum_{i=1}^n \einsfun\{\g X_i\in  R^{X}_{n-k_n,n}A_j\}.
\end{align}
The empirical estimates $\hat q_j$ of $q_j$, $j=1,\dots, K$, are defined in a similar way based on the samples of $\g Y$. 
By plugging in these probability estimates into the definition of $D_K$ in~\eqref{eq:KL_Multinomial}, we obtain an empirical estimator of the \KL{} divergence by 
\begin{equation}\label{eq:KL_Multinomial2}
  \Hat{D}_K = \Hat{D}_K (n):= \sum_{j=1}^K (\hat p_j-\hat q_j)(\log \hat p_j-\log \hat q_j).
\end{equation}
 The latter can be compared to the quantiles of the limiting chi-squared distribution of $\hat D_K$ under the null hypothesis, as $n\to\infty$, to obtain critical values for the test statistic; we refer to Section~\ref{sec:known marginals} for the derivation of the asymptotic distribution.

In practice, we have i.i.d.~samples drawn from the original random vectors $\g{\tilde X}$ and $\g{\tilde Y}$, denoted as $\g{\tilde X}_1,\dots,\g{\tilde X}_n$ and $\g{\tilde Y}_1,\dots,\g{\tilde Y}_n$. Consequently, to obtain i.i.d.~samples drawn from $\g X$ and $\g Y$, the marginal distributions of the random vectors $\g{\tilde X}$ and $\g{\tilde Y}$ must be known. If the marginal distributions are unknown, they can be be estimated either parametrically or non-parametrically from the data. Since our focus is on testing extremal dependence, and in order to avoid parametric assumptions, we use empirical estimates. Write $\g{\tilde X}_i=(\tilde X_{i1},\dots, \tilde X_{id})$ and denote by $\hat F_j$ the empirical distribution function based on the samples $\{\tilde X_{ij}\}_{i=1}^n$. We define new samples $\g{\hat X}_i$, $i=1,\dots, n$, with entries
\begin{align}\label{X_pseudo}
  \hat X_{ij} = 1 / \{1 - \hat F_j(X_{ij})\} = \frac{n+1}{n+1-\mbox{rank}(\tilde X_{ij})}, \quad j = 1,\dots, d,
\end{align}
where $\mbox{rank}(\tilde X_{ij})$ denotes the rank of $\tilde X_{ij}$ in the sample $\{\tilde X_{kj}\}_{k=1}^n$. Similarly 
we define $\g{\hat Y}_i$, $i=1,\dots, n$ for the second sample.
The samples are called pseudo-observations since $\g{\hat X}_1, \dots ,\g{\hat X}_n$ are no longer independent. This is due to the fact that the marginal distributions are estimated from the same data set. In addition, the common marginal distribution is not exactly a standard Pareto distribution due to the deterministic values of the ranks.

We subsequently obtain the  \KL{} test statistic $\hat D_K$ by computing the empirical probabilities in~\eqref{p_estimate} using these pseudo-observations and then plugging them into the definition of the divergence in~\eqref{eq:KL_Multinomial2}. Because of the additional step of standardizing the marginal distributions empirical, the limiting distribution of $\hat D_K$ is more complicated and can in general only be assessed by a bootstrap procedure; see Section~\ref{sec:bootstrap} for details.

Algorithm~\ref{alg:KL} summarizes the different steps of our \KL{} divergence extremal dependence test.

\begin{algorithm}
  \caption{\KL{} divergence test}
  \label{alg:KL}
  
  \textbf{Input:} Two data sets $\g{\tilde X}_1,\dots, \g{\tilde X}_n$ and $\g{\tilde  Y}_1,\dots, \g{\tilde Y}_n$; a risk functional $r$; partitioning sets $A_1,\dots ,A_K$; number of exceedances $k_n$ to be used; significance level $\alpha \in (0,1)$.

  \textbf{Output:} The test statistic $\hat D_K$, the $p$-value and whether the null hypothesis is rejected.
 
  \begin{algorithmic}[1]
    \Procedure{\KL{} divergence test}{$k_n$, $\alpha$}
    \State Standardize the margins: 
    \begin{itemize}
      \item[a)] if the margins are known, use~\eqref{X_true} for both data sets;
      \item[b)] if the margins are unknown, obtain pseudo-observations through~\eqref{X_pseudo} for both data sets using empirical distribution functions and use these transformed samples.
    \end{itemize} 
    \State For each data set, select $k_n$ exceedances and compute probability estimates $\hat p_j$ and $\hat q_j$ according to~\eqref{p_estimate}.
    \State Compute the observed value $\hat D_K$ of the \KL{} divergence in~\eqref{eq:KL_Multinomial} by plugging in the $\hat p_j$ and $\hat q_j$. 
    \State Compute the $p$-value using the approximation of the null distribution:
    \begin{itemize}
      \item[a)] if the margins are known
      \[p = \mathbb P\left(\chi^2(K-1) > k_n \hat D_K/2\right); \]
      \item[b)] if the margins are unknown, the bootstrap procedure described in Section~\ref{sec:bootstrap} yields $B$ bootstrapped versions $\hat D^{(1)}_K, \dots, \hat D^{(B)}_K$ of the test statistic under $H_0$, and 
      \[p = \frac{1}{B} \sum_{b=1}^B \einsfun\{\hat D^{(b)}_K > \hat D_K\}. \]
    \end{itemize} 
    \State \textbf{Return:} Value of $\hat D_K$, the $p$-value and "reject $H_0$" if $p<\alpha$, and "not reject $H_0$" otherwise. 

    \EndProcedure
  \end{algorithmic}
\end{algorithm}

\subsection{Theoretical guarantees with known marginal distributions}\label{sec:known marginals}

In this section we derive theoretical guarantees for the asymptotic distribution of the \KL{} divergence test statistic $\hat D_K$ defined in Section~\ref{sec:method} as the sample size $n\to\infty$. We first work under the assumption of known marginal distributions such that {we can obtain i.i.d.~samples drawn from the distributions of $\g X$ and $\g Y$, which are used to construct the test statistic $\hat D_K$.} We assume that $\g X$ and $\g Y$ satisfy the MRV condition~\eqref{mevd}, with tail indices $\alpha_X$ and $\alpha_Y$ and exponent measures $\nu_X$ and $\nu_Y$, respectively.

In order to show asymptotic properties of $\hat D_K$ we need to 
state some technical conditions on the risk functional, the possible partitions and the distributions of $\g X$ and $\g Y$.
The risk functional $r$ should be as in Section~\ref{sec:background} with the additional restriction that $\nu_X(\partial \Omega_r)=\nu_Y(\partial \Omega_r)=0$, where 
$\Omega_r=\set{\g x \in[0,\infty)^d: r(\g x)>1}$.
This assumption is satisfied if the exceedance distribution has a density, which is the case in most parametric models \citep[e.g.,][]{col1991, eng2014}.

In order to control the error of the probability estimates $\hat p_j$ compared to their theoretical limit, we can decompose
\begin{align*}
  \left| \hat p_j - \nu_X(A_j)/\nu_X(\Omega_r) \right| &\leq  \left|\hat p_j - P(\g X\in u_n^XA_j\mid r(\g X)>u_n^X)\right|\\
  & + \left|P(\g X\in u_n^XA_j\mid r(\g X)>u_n^X) -\nu_X(A_j)/\nu_X(\Omega_r) \right|.
\end{align*}
The second error term is the deterministic approximation bias related to the fact that the threshold $u_n$ is finite. This error typically appears in extreme value statistics and disappears as the threshold grows under the following standard second-order condition. We assume that there exits a second-order scale function $a_X$ such that, as $u\to\infty$, $a_X(u)\to 0$ and for any Borel set $B\subset [0,+\infty)^d\setminus [0,\epsilon)^d$ with $\nu_X(\partial B)=0$,
\begin{align}\label{mevd_2nd}
 \mathbb P(\g X/u \in  B|\norm{\g X}_\infty >u) -\nu_X(B) = O(a_X(u)).
\end{align}

In order to define a large class of admissible partitions $A_1,\dots, A_K$ of the set $\Omega_r$ we make the following assumption.
Let $\mathcal B(\Omega_r)$ denote all Borel subsets of $\Omega_r$ and let
\begin{align}\label{stable_sets}
  \mathcal A_r = \{B \subset \mathcal B(\Omega_r): t B \subset B \text{ for all } t>1\}
\end{align}
be all Borel subsets of $\Omega_r$ whose inflation lies again in the same set. 

\begin{assumption}\label{ass:sets}
  Let $r$ be a homogeneous and measurable risk functional. 
  Let $A_1,\dots, A_K$ be a partition of $\Omega_r = \{\g x \in [0,\infty)^d :r(\g x) > 1\}$ such that for all $j=1,\dots, K$ the set $A_j$ can be expressed using a finite number of the operations of union, intersection, and complementation on the sets in $\mathcal A_r$.
\end{assumption}

We can easily check that the examples in Section~\ref{sec:background} all satisfy this assumption. Indeed, the sets of the partition in Example~\ref{example: radius} are already contained in $\mathcal A_r$. In Example~\ref{example: max}, each of the sets $A_I$ can be obtained as a finite intersection of sets of the form $B_I= \{x\in\mathbb R^d: x_j>1, j\in I\} \in\mathcal A_r$.
Similarly, in Example~\ref{example: min}, the sets $A_I$ are finite intersections of sets of the form $B_I= \{x\in\mathbb R^d: x_j>2, j\in I\} \in\mathcal A_r$.
More generally, Assumption~\ref{ass:sets} is very mild and it is satisfied for most practically relevant examples.

The theorem below shows the asymptotic behavior of the test statistic $\hat D_K$ under both the null and the alternative hypotheses.

\begin{theorem}\label{thm_AD_hom}
  Suppose that for random vectors $\g X$ and $\g Y$, the condition \eqref{mevd_2nd} holds with exponent measures $\nu_X$ and $\nu_Y$, and second order scale functions $a_X$ and $a_Y$, respectively. Assume that Assumption~\ref{ass:sets} holds for the risk functional $r$ and partition $A_1,\dots, A_K$ of the set $\Omega_r$, with $\nu_X(\partial A_j)=\nu_Y(\partial A_j)=0$ for $j=1,\ldots,K$.
  Define $U_{X}(t)=\inf\set{u: \mathbb P(\norm{\g X}_\infty >u)\leq 1/t}$ and $U_{Y}$ accordingly. 
  Let $k_n$ be a sequence such that as $n\to\infty$, $k_n\to\infty$, $k_n/n \to 0$, $\sqrt{k_n} a_X\{U_X(n/k_n)\}  \to 0$ and $\sqrt{k_n} a_Y\{U_Y(n/k_n)\}  \to 0$.
  
  Under the null hypothesis $H_0$ in~\eqref{H0}, the Kullback--Leibler divergence converges in distribution to a chi-squared distribution with $K-1$ degrees of freedom, that is,
\begin{align}\label{rate_null_known}
\frac{k_n}{2} \hat D_K  \stackrel{d}{\to} \chi^2(K-1),\quad  n \to \infty.
\end{align}
  Under the alternative hypothesis that $p_{j} \neq q_{j}$ for some $j=1,\ldots, K$. Then we have that
$$
\sqrt{k_n} \suit{\hat D_K -D_K} \stackrel{d}{\to}  N(0, \sigma^2), \quad n \to \infty,
$$
where $D_K$ is defined in \eqref{eq:KL_Multinomial} and 
$$\sigma^2=\sum_{j=1}^K D_{1j}^2p_{1}(1-p_{1})+D_{2j}^2q_{1}(1-q_{1})-2\sum_{1\leq j_1<j_2\leq K}\suit{D_{1j_1}D_{1j_2}p_{j_1}p_{j_2}+D_{2j_1}D_{2j_2}q_{j_1}q_{j_2}}$$
with
$$D_{1j}=\frac{\partial D_K}{\partial p_j}=\log \frac{p_j}{q_j}+1-\frac{q_j}{p_j}, \mbox{\ and\ }D_{2j}=\frac{\partial D_K}{\partial q_j}=\log \frac{q_j}{p_j}+1-\frac{p_j}{q_j},$$
for $1\leq j\leq K$.
\end{theorem}
The theorem leads to the asymptotic justification of the $p$-value used in Algorithm~\ref{alg:KL} under the null hypothesis and known marginals; see 5a) in the Algorithm. Under the alternative, since $\hat D_K\tp D_K>0$, we get that $\frac{k_n}{2}\hat D_K\tp \infty$. Therefore, the $p-$value obtained from the Algorithm \ref{alg:KL}, 5a) converges to 1 as $n\to\infty$. This shows the asymptotic power of the KL divergence test.

\subsection{Theoretical guarantees with unknown marginal distributions}\label{sec:unknown marginals}

In this section we consider the case that the marginal distributions of $\g{\tilde X}$ and $\g{\tilde Y}$ are unknown. As {discussed} in~\eqref{X_pseudo}, the data {are first} normalized to a common approximate Pareto marginal scale {by using empirical distributions for each marginal}. 
We then apply our KL divergence test to the transformed pseudo observations.

Using the pseudo-observations has an impact on the asymptotic behavior of the probability estimates $\hat p_j$ and $\hat q_j$, which eventually affects the asymptotic behavior of the KL divergence statistic. 
Assuming generically that the estimators $\hat p_j$ and $\hat q_j$ are asymptotically normal, the following theorem shows the asymptotic behavior of the test statistic $\hat D_K$ under both null and alternative hypothesis. While the limiting distribution does not have an explicit form as in the previous section, the asymptotic theory provides a theoretical justification for applying a subsample bootstrap procedure to obtain critical values of the test; see Section~\ref{sec: sim} for details.

\begin{theorem} \label{thm_generic}
  For a suitable $k_n$, assume that as $n\to\infty$,
  $$\sqrt{k_n}(\hat p_j-p_j)=N_j^X+o_P(1) \text{\ \ and\ \ } \sqrt{k_n}(\hat q_j-q_j)=N_j^Y+o_P(1), $$
  for all $1\leq j\leq K$, where $(N_1^X,\ldots, N_K^X)$ and  $(N_1^Y,\ldots, N_K^Y)$ are two independent normally distributed random vectors with mean zero and covariance matrices $\mathbf{\Sigma}^X$ and $\mathbf{\Sigma}^Y$. 

  Under the null hypothesis $H_0$ in~\eqref{H0}, the Kullback--Leibler divergence converges in distribution 
\begin{align}\label{rate_null}
\frac{k_n}{2} \hat D_K  \td \sum_{j=1}^K \frac{1}{2p_j}(N^X_j-N^Y_j)^2, \quad n\to\infty.
\end{align}

Under the alternative hypothesis that $p_{j} \neq q_{j}$ for some $1\leq j\leq K$, denote by $D_K$ the population version in~\eqref{eq:KL_Multinomial} with $p_{j},q_{j}$. Then we have that
$$
\sqrt{k_n} \suit{\hat D_K -D_K} \td \sum_{j=1}^K\suit{D_{1j}N^X_j+D_{2j}N^Y_j}, \quad n\to\infty,
$$
with $D_{ij}$ defined as in Theorem \ref{thm_AD_hom} for $i=1,2$ and $1\leq j\leq K$. 
\end{theorem}

The conditions in Theorem \ref{thm_generic} are generic and the limit distributions under both the null and alternative hypothesis remain inexplicit and dependent on the unknown parameters of extremal dependence.
In particular, even under the  null distribution, the asymptotic limit is no longer a simple chi-squared distribution as in the case of known marginal distributions. In Section~\ref{sec:bootstrap}, we therefore propose a bootstrap procedure to obtain critical values for the test under the null hypothesis.

Below we verify for the three examples from the previous sections that all assumptions of the above theorem are satisfied, under certain regularity conditions, for the empirical estimators $\hat p_j$ and $\hat q_j$ based on the pseudo-observations in~\eqref{X_pseudo}.

\begin{example}\label{example: max_2}
  Consider the risk functional $r(\g x)=\max(x_1, \dots, x_d)$ and
  the partition composed of sets $A_I=\{x\in\mathbb R^d: x_j>1, j\in I; x_j\leq 1, j\in I^c\}$ for any non-empty $I\subset \{1,\dots, d\}$; see Example~\ref{example: max}. The validity of the assumptions on the asymptotic normality in Theorem~\ref{thm_generic} of the empirical estimators $\hat p_j$ and $\hat q_j$ is related to the estimation of the tail dependence functions; see, e.g. Chapter 7 in \cite{deh2006}. We also provide a specific example for $d=2$ in the Appendix.
\end{example}

\begin{example}
  Consider the risk functional $r(\g x)=\min(x_1, \dots, x_d)$  and
  the partition composed of sets $A_I=\{x\in\mathbb R^d: x_j>2, j\in I; 1<x_j\leq 2, j\in I^c\}$ for any non-empty $I\subset \{1,\dots, d\}$; see Example~\ref{example: min}. Then the validity of the assumptions on the asymptotic normality in Theorem~\ref{thm_generic} of the empirical estimators $\hat p_j$ and $\hat q_j$ is also related to the estimation of the tail dependence function.
\end{example}

\begin{example}
  Consider a specific case for Example \ref{example: radius} in $d=2$ dimensions where $r$ is the $L_2$ norm $r(\g x)=\sqrt{x_1^2+x_2^2}$, accompanied with the angular component $\theta(\g x)=\arctan(x_2/x_1)$. Let $0=\theta_0<\theta_1<\ldots<\theta_K={\pi}/{2}$ be a sequence of angles. We divide the unit sphere $S_r$ into $K$ regions according to these angles, which then results into a partition of the set $\Omega_r=\set{\g x \geq 0: x_1^2 + x_2^2 >1}$  
  \[A_j=\set{\g x=(x_1,x_2): r(\g x)>1, \theta(\g x)\in(\theta_{j-1},\theta_{j}]}, \quad j=1,\ldots, K,\]
  which is admissible in the sense of Assumption~\ref{ass:sets}.  We can verify that the assumptions of Theorem~\ref{thm_generic} are satisfied for the empirical estimators $\hat p_j$ and $\hat q_j$ under some additional mild conditions; see  \cite{einmahl2001}.
\end{example}

In general, the asymptotic distributions of $\hat D_K$ under the null and alternative hypothesis do not have a simple closed form. Nevertheless, in the special case of $d=2$ in Example~\ref{example: max_2}, we derive such an explicit form; see Appendix~\ref{sec:max_fct} for details. In particular, it is interesting to note that the limit under the null distribution is a scaled chi-squared distribution whose scaling constants depend on the unknown exceedances distributions $\nu_X$ and $\nu_Y$.

\section{Simulations} \label{sec: sim}

We perform several simulation studies to assess the performance of the \KL{} divergence test for the equality of extremal dependence structures between two samples; see Algorithm~\ref{alg:KL} for details on the implementation. We investigate different aspects such as the sensitivity to the choice of the risk function and the partition, the effect of empirical marginal standardization and the type of alternative hypothesis.

We consider two different copula families. The first is the class of logistic distributions with parameter $\theta \in (0,1]$ \citep{gum1960}, which is part of the class of extreme value copulas \citep{seg2010}. Its distribution function on uniform margins is 
\begin{align}\label{log_cop}
  F_\theta^{\text{log}}(x_1,x_2) = \exp\left\{ - \left[ (-\log x_1)^{1/\theta} + (-\log x_2)^{1/\theta} \right]^\theta \right\}.
\end{align}
The second is the class of outer power Clayton copulas, which also has a parameter $\theta\in (0,1]$, and distribution function
\begin{align*}
  F_\theta^{\text{clay}}(x_1,x_2) = \left\{ \left[ (x_1^{-1} -1)^{1/\theta} + (x_2^{-1} -1)^{1/\theta} \right]^\theta + 1 \right\}^{-1}.
\end{align*}
It can be checked that after transformation to standard Pareto margins, both the logistic and the outer power Clayton copula are multivariate regular varying and satisfy~\eqref{mevd} with the same extended exceedance distribution
$$ \nu_\theta([\mathbf{0}, \boldsymbol{\infty}] \setminus [\mathbf{0},\g x]) \propto \left( x_1^{1/\theta} + x_2^{1/\theta} \right)^\theta, \qquad \g x \in [0,\infty)^2 \setminus \{0\}$$
Importantly, even if the dependence parameters $\theta$ are equal for both models, the two copulas $F_\theta^{\text{log}}$ and $F_\theta^{\text{clay}}$ are still different; they then only coincide in the extreme tails.

In the sequel, we use the R package \texttt{copula} \citep{copula2023} to generate independent samples from
the logistic and the outer power Clayton copula.

\subsection{Bootstrap null distribution}\label{sec:bootstrap}

As mentioned in Section~\ref{sec:unknown marginals}, when the marginal distributions are unknown, the null distribution of the test statistic $\hat D_K$ does not have a simple closed form. We therefore obtain critical values using a bootstrap procedure.
To this end, we randomly sample $n/2$ (assuming that $n$ is even) of the observations $\g X_1, \dots, \g X_n$ without replacement. We then compute the divergence between this sample and the remaining $n/2$ samples. We repeat this $B$ times, where $B$ is the desired number of bootstrap samples. By construction, this results in $B$ samples $\hat D^{(1)}(n/2), \dots, \hat D^{(B)}(n/2)$ of the test statistic $D_{K}(n/2)$ under the null hypothesis. Since the sample size is only half as large as the desired sample size $n$, we correct for this setting
\[ \hat D^{(i)}(n)  =\hat  D^{(i)}(n/2) / 2, \quad i = 1,\dots, B, \]
as suggested by the convergence rate in Theorem~\ref{thm_generic}.

We illustrate the effectiveness of the bootstrap approach in a simulation study, where $\g X_1, \dots, \g X_n$ are generated from the outer power Clayton copula with $n=2000$.
The divergence is computed for the risk functional $r(\g x) = \sqrt{x_1^2 + x_2^2}$ with partitions as in Example~\ref{example: radius} comprised of $K=4$ sets, and the number of exceedances is set to $k_n = 200$.
To obtain the bootstrap null distribution, we resample the data as described above with $B=1000$. To obtain samples from the true null distribution, we simulate $1000$ new data sets from the same outer power Clayton distribution and compute the divergences.
The left-hand side of Figure~\ref{fig:histo} compares the histograms of the bootstrap and the true null distribution, properly normalized according to~\eqref{rate_null_known}, when the margins are assumed to be known. We see that they match well. They also agree nicely with the density of the theoretical chi-squared limit (black line). When the margins are unknown and have to be normalized empirically, we still know the rate of convergence by~\eqref{rate_null}, but the limiting null distribution does no longer agree with the chi-squared distribution. Also in this case, the bootstrap samples approximate well the null distribution. Therefore, critical values can be computed from this bootstrap distribution instead of the chi-squared quantiles. 

\begin{figure}[tb!]
	\centering
	\includegraphics[width=.9\linewidth]{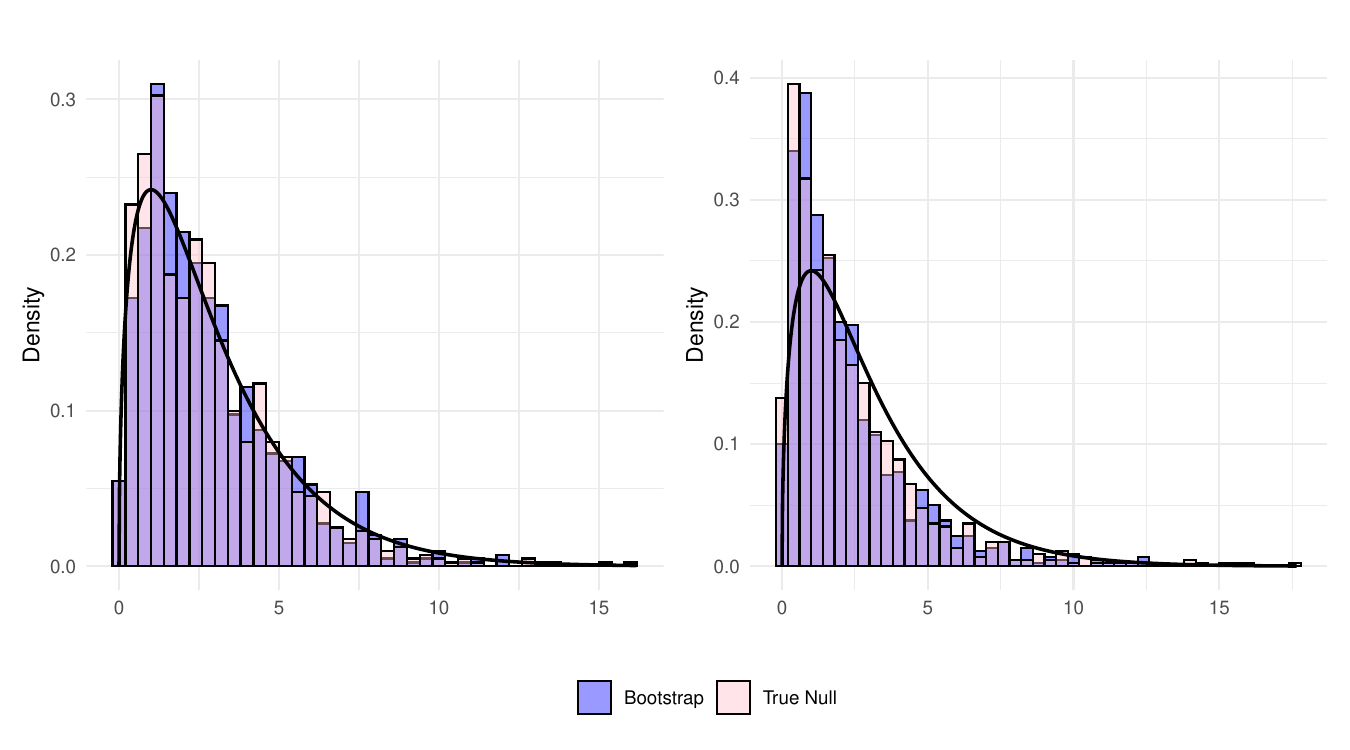}
	\caption{Histograms of bootstrap distribution (blue) and true null distribution (pink) together with the density of a chi-squared distribution with $K-1=3$ degrees of freedom, for the situation when the margins are known (left), and when the margins are unknown and have to be normalized empirically (right).}
	\label{fig:histo}
\end{figure}

\subsection{Power of the test}

We next investigate the power of our test in different 
scenarios. We generate $n=2000$ independent samples $\g X_1, \dots, \g X_n$ and $\g Y_1, \dots, \g Y_n$ from the outer power Clayton copula with parameters $\theta_X$ and $\theta_Y$, respectively.
 We compute the test statistic $\hat D_K$ in~\eqref{eq:KL_Multinomial} for two risk functionals $r(\g x) = \max(x_1,x_2)$ and $r(\g x) = \sqrt{x_1^2 + x_2^2}$ with partitions as in Examples~\ref{example: max} and~\ref{example: radius} comprised of $K=3$ and $K=5$ sets, respectively; see also Figure~\ref{fig:risks}. For the threshold $u_n$, we choose a sequence of different numbers of upper order statistics $k_n$ to be used as exceedances of the risk functional $r$.

We start with the case of known margins where standardization~\eqref{X_true} can be used.
Theorem~\ref{thm_AD_hom} then states that the asymptotic distribution of the test statistic is $\chi^2(K-1)$, where $K$ is the number of sets in the partition. 
The top left panel of Figure~\ref{fig:clay_clay_H0} shows the summary of
a simulation study with 500 repetitions for the case $\theta_X = \theta_Y = 0.45$, that is, when the null hypothesis~\eqref{H0} holds.
The solid and dashed blue lines represent the mean of the 500 KL test statistic values $\hat D_K$ and the corresponding pointwise empirical $5\%$ and $95\%$ quantiles for each $k_n$, respectively. The critical values of the $\chi^2(K-1)$ distribution at level $95\%$ as a function of the number of exceedances $k_n$ are shown by the dashed black line. As expected by the theory, approximately $95\%$ of the realizations of the test statistic stay below the critical level. This is also confirmed by the orange line, representing the percentage of times when $H_0$ is rejected.
Throughout all values of $k_n$, the rejection rate is approximately $5\%$, showing that the test has the desired level.

The top center panel of Figure~\ref{fig:clay_clay_H0} shows the same figure as before, but now for data simulated under the alternative hypothesis that $0.45 = \theta_X \neq \theta_Y = 0.55$. We see ability of the test to distinguish dependence structures as most of the $\hat D_K$ values fall above the critical value. Even though the dependence structures of $\g X$ and $\g Y$ are not too different, the orange line shows that the rejection probability and thus the power of the test is high.

An even more interesting situation occurs if we simulate the samples $\g X_1, \dots, \g X_n$ as before from the outer power Clayton copulas, but the samples $\g Y_1, \dots, \g Y_n$ of the second population are generated from the logistic distribution~\eqref{log_cop}. In particular, we set the dependence parameters of both models to the same value $\theta_X = \theta_Y = 0.45$.
The null hypothesis~\eqref{H0} then holds true since the extremal dependence structures coincide. In the top right panel of Figure~\ref{fig:clay_clay_H0} we observe that the test is rejected for large values of $k_n$ (low thresholds $u_n$), and no longer rejected for small enough $k_n$. This shows that the test focuses on extremal dependence; this also agrees with the theory that requires that $k_n/n \to 0$.

We repeat all simulations for the situation when the marginal distributions are unknown and have to normalized empirically according to~\eqref{X_pseudo}. We use critical values obtained by the bootstrap procedure described in Section~\ref{sec:unknown marginals}.  The bottom row of Figure~\ref{fig:clay_clay_H0} shows that the results remain similar to the case of known margins, confirming that the bootstrap procedure is effective.

We repeat all simulations for
the maximum risk $r(\g x) = \max(x_1,x_2)$, 
which has $K=3$ sets.
Figure~\ref{fig:clay_clay_H0_max}
in Appendix~\ref{app:simu} shows qualitatively similar results,
but the power of the test is generally
lower than that for the Euclidean risk
functional with $K=5$ sets.

\begin{figure}[htbp]
	\centering
	\hspace*{-2.7em}\includegraphics[width=0.38\linewidth]{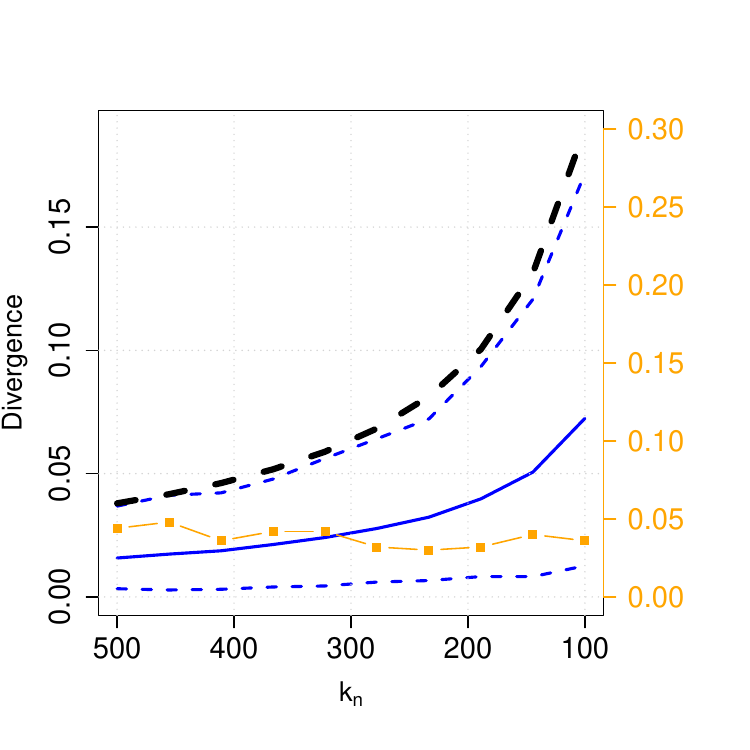}
  \hspace*{-1.8em}\includegraphics[width=0.38\linewidth]{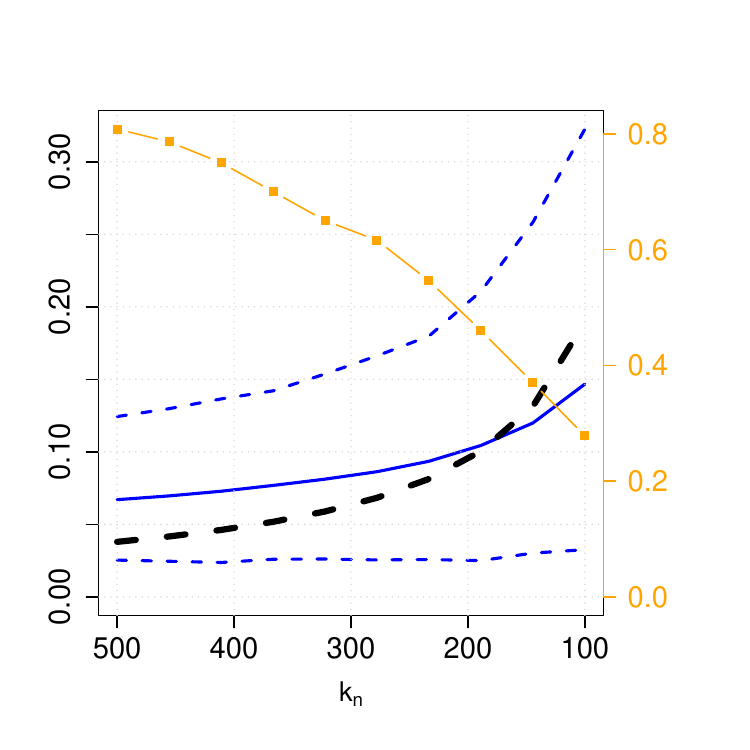}
	\hspace*{-1.8em}\includegraphics[width=0.38\linewidth]{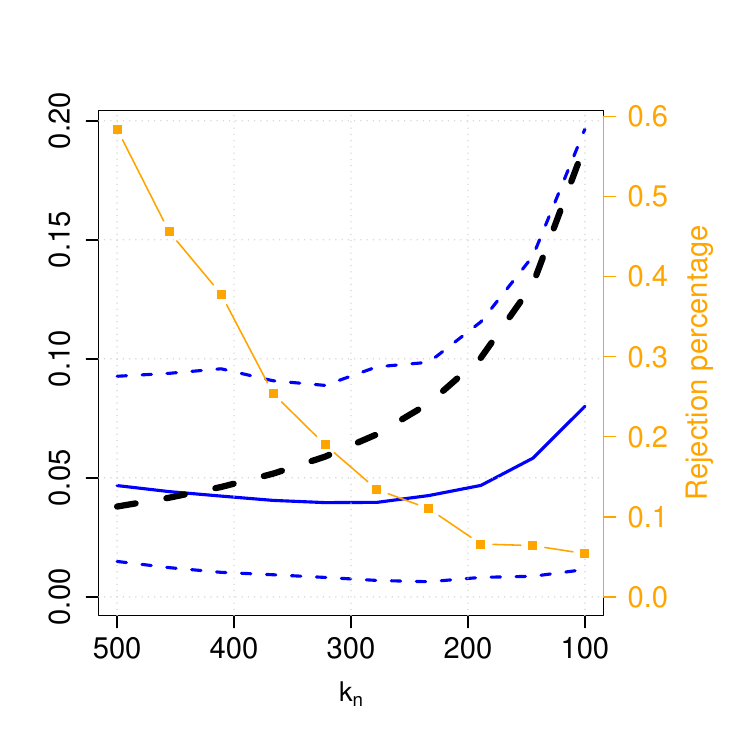}
 \hspace*{-3em}\\ \vspace*{-4.5em}%
	\hspace*{-2.7em}\includegraphics[width=0.38\linewidth]{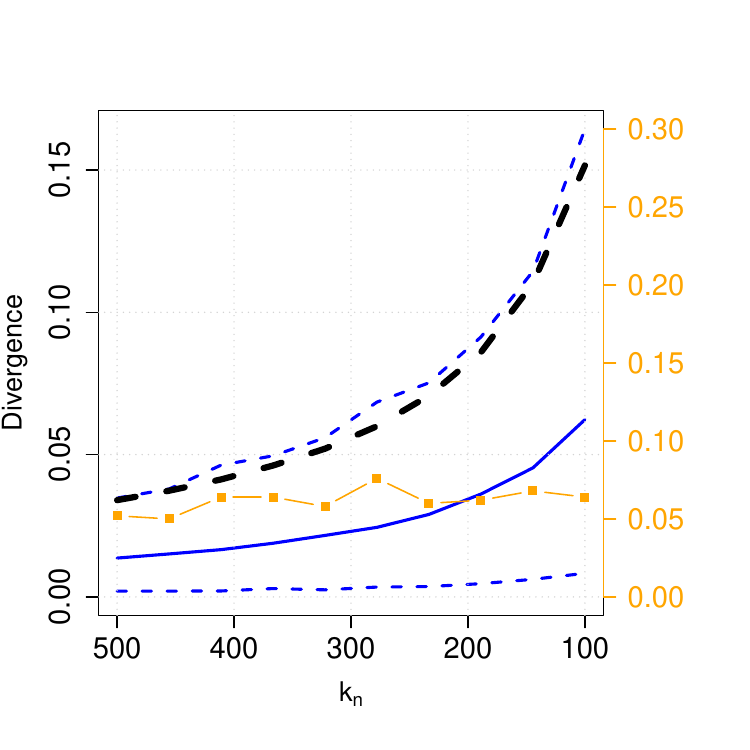}
  \hspace*{-1.8em}\includegraphics[width=0.38\linewidth]{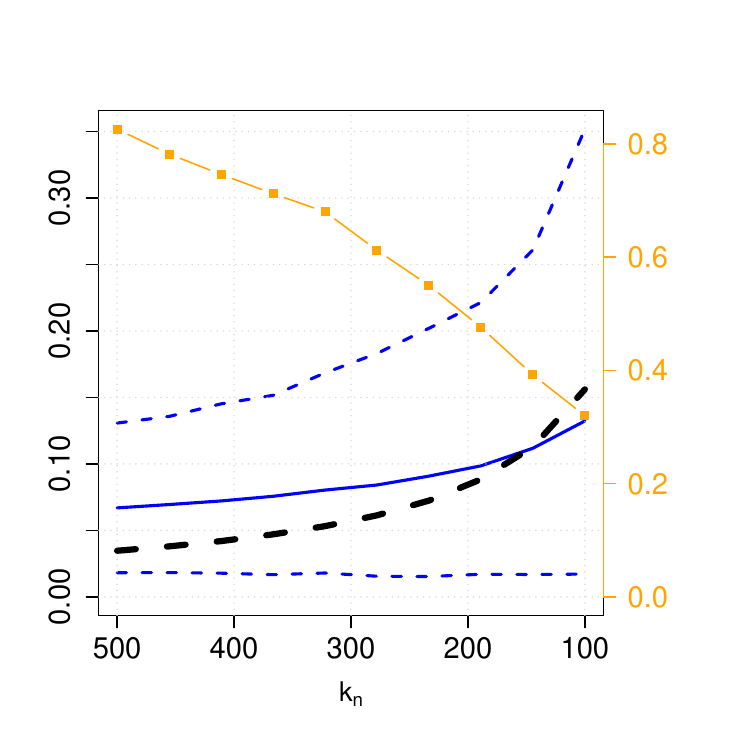}
	\hspace*{-1.8em}\includegraphics[width=0.38\linewidth]{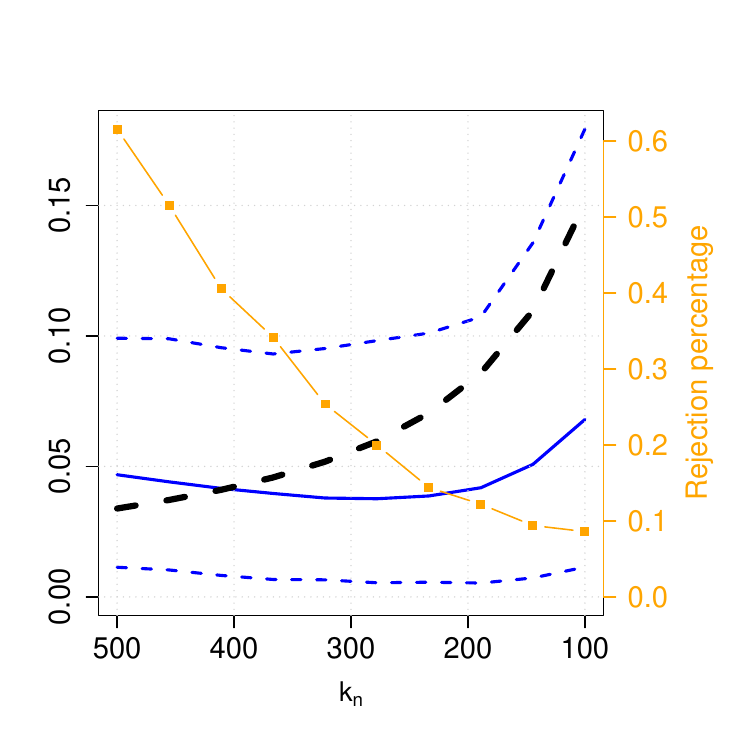}
 \vspace*{-2em}
  \caption{Mean (solid blue line) and empirical $5\%$ and $95\%$ quantiles (dashed blue lines) of 500 samples of KL test statistic $\hat D_K$ based on Euclidean risk functional with $K=5$ sets for a range of exceedances $k_n$
  together with the critical values (black dashed line) at level $95\%$ and rejection percentages (orange line). Top and bottom rows show results for known and unknown margins, respectively. 
   The two samples are generated from the same distribution (left), from distributions with different extremal dependence structure (center), and from different distributions with same extremal dependence structure (right).}
 \label{fig:clay_clay_H0}
\end{figure}

\subsection{The role of $K$}

We now study the effect of the number $K$ of sets in the partitions for the example of  the Euclidean risk functional $r(\g x) = \sqrt{x_1^2 + x_2^2}$ with partitioning as in Example~\ref{example: radius}. As an additional comparison, we also use the sum risk functional $r(\g x) = |x_1| + |x_2|$ with a similar partitioning along the angles and different numbers of sets. As a baseline, we consider the maximum risk functional $r(\g x) = \max(x_1,x_2)$ from Examples~\ref{example: max}, which has a constant number of sets $K=3$.
First, we simulate $n=5000$ samples from two outer power Clayton copulas with known margins and different dependence parameters $0.45 = \theta_X \neq \theta_Y = 0.55$. For $k_n=200$ exceedances, the left-hand side of Figure~\ref{fig:nbsets} shows the rejection probabilities over 500 repetitions as a function of $K$ for the different risk functionals. We first observe that both angular partitionings cannot distinguish the two distributions for $K=2$ sets. The reason is that $\g X$ and $\g Y$ are both symmetric and therefore an (approximately) equal number of points fall in the upper and lower halves. When $K$  increases, the power of the test increases rapidly and stabilizes above for 4 or 5 sets. Importantly, it exceeds substantially the power of the simple test based on the maximum risk. This shows that despite its simplicity, our method is sufficiently flexible to improve over more basic statistical tests.        
The choice of $K$ amounts to a bias-variance trade-off. For too small $K$ there might be a bias in comparing the two extremal dependence structures in $\g X$ and $\g Y$ since the difference might not be visible on the few sets. If $K$ is too large, then every set contains only few observations and there is a high variability in estimating the divergence $\hat D_K$.

Secondly, using the R package \texttt{grahpicalExtremes} \citep{graphicalExtremes2024} following the exact simulation method in \cite{dom2016}, we generate $\g X$ from the asymmetric Dirichlet distribution introduced in \cite{col1991}, and compare it to samples of~$\g Y$ from the outer power Clayton copula. We choose the parameters of each model such that they have the same extremal correlation $\chi_X = \chi_Y$ \citep{col1999, sch2003}, a popular summary statistic for extremal dependence. In this case, $H_0$ does not hold since the extremal dependence structures differ in terms of their symmetry properties.
For the maximum risk functional, we can rewrite the test statistics $\hat D_K$ as a simple function of the extremal correlation; see Appendix~\ref{extremal_correlation}. Since both samples have the same $\chi$ coefficient, the right-hand side of Figure~\ref{fig:nbsets} shows that the maximum risk cannot distinguish the two extremal dependence structure (even though they are different).
Both angular risks detect this difference in symmetry already for a small number of sets. We see also here an increasing performance, even up to $K=8$ sets.

\begin{figure}[tb!]
	\centering
	\hspace*{-2em}\includegraphics[width=0.44\linewidth]{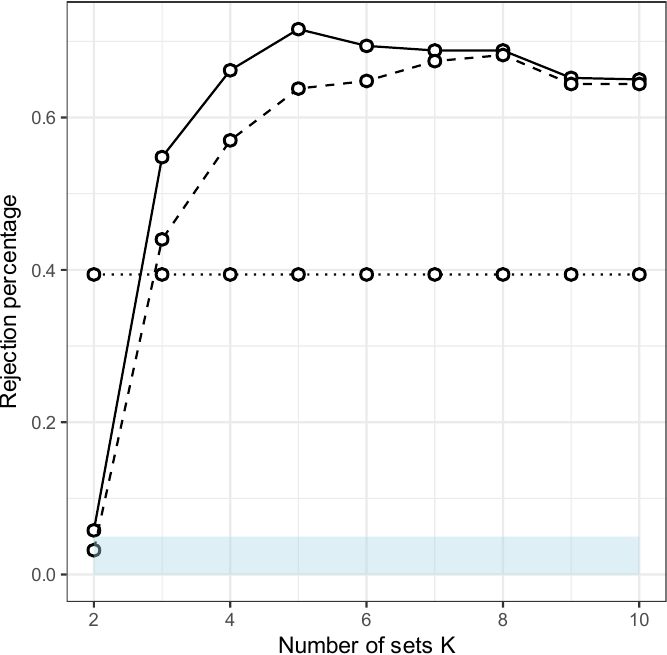}
  \hspace*{2em}\includegraphics[width=0.44\linewidth]{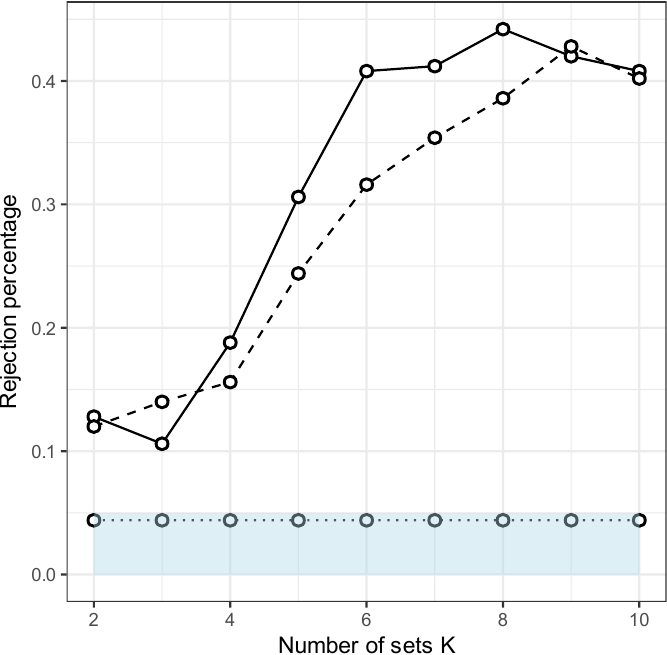}
  \caption{Rejection percentages of $H_0$ for 500 simulations of Clayton vs Clayton (left) and Dirichlet vs Clayton (right) for different choices of risk functionals: the maximum risk (dotted line), the Euclidean norm (solid line) and the sum risk (dashed line). Blue shaded area indicates significance level.}
	\label{fig:nbsets}
\end{figure}

\section{Application}\label{sec: Application}
During the last twenty years, some national weather services have consistently measured  sub-hourly precipitation at  a country scale  with  reliable networks of  weather  gauges. Such time series  are important for understanding and predicting flash floods, urban drainage issues, and assessing the impacts of climate change on extreme rainfall events. 
A key aspect is to understand the relationship between precipitation intensity, duration, and frequency. Visually, an intensity-duration-frequency (IDF) curve can be obtained for a single location and this curve provides a commonly used tool for the design of hydrological structures \citep{Ulrich20}.
Statistically, IDF curves can be viewed as a summary of the marginals behaviors of aggregated rainfall  at different time scales (minutes, hours, weeks, etc.). Such marginals are fitted by parametric models such as the generalized extreme value \citep{Ulrich20} or the extended Pareto distribution \citep{haruna2023modeling}. 
An important aspect is the role of the seasonal cycle: 
storm intensity in Europe, for instance, varies with the season due to shifts in atmospheric conditions and temperature differences between seasons. Specifically, warmer temperatures in summer and fall can lead to more intense storms, particularly in the form of increased rainfall. Conversely, winter storms are likely  to  be driven by cold air masses interacting with warmer air from the Atlantic. 
For Germany, \cite{ulrich2021modeling} modeled how  seasonal variations of extreme
rainfall change according to  different aggregation timescales.
Again, such studies focused on marginal changes, but it would be of hydrological interest to assess if the seasonal cycle also impacts the extremal dependence among different aggregation scales.  
We will show in this section that our KL divergence offers a simple tool to answer such a question.

We study rainfall time series from the French weather service M\'et\'eo-France that provides a public database\footnote{\url{https://meteo.data.gouv.fr/datasets/donnees-climatologiques-de-base-6-minutes/}} of 6-minute rainfall recordings over the period 2006--2023. 
To avoid short  sample lengths, we focus here on daily maxima of 6-minute and hourly rainfall for four seasons:  winter (DJF), spring (MAM), summer (JJA) and fall (SON); 
see the bivariate samples in Figure~\ref{fig:bordeaux} for the city of Bordeaux in winter (blue) and spring (orange). 
To compare extremal dependence between heavy rainfall at these two time scales,
Figure~\ref{fig:bordeaux4sets} shows our test statistic $\hat D_K$ in~\eqref{eq:KL_Multinomial2} as a function of the number of exceedances $k_n$ for all combinations of seasons;
here we fix $K=4$, but the results are similar with other choices.
\begin{figure}[tb]
\centering
	\hspace*{-2.7em}\includegraphics[trim={0cm 2cm 0 6cm},clip,width=.38\linewidth]{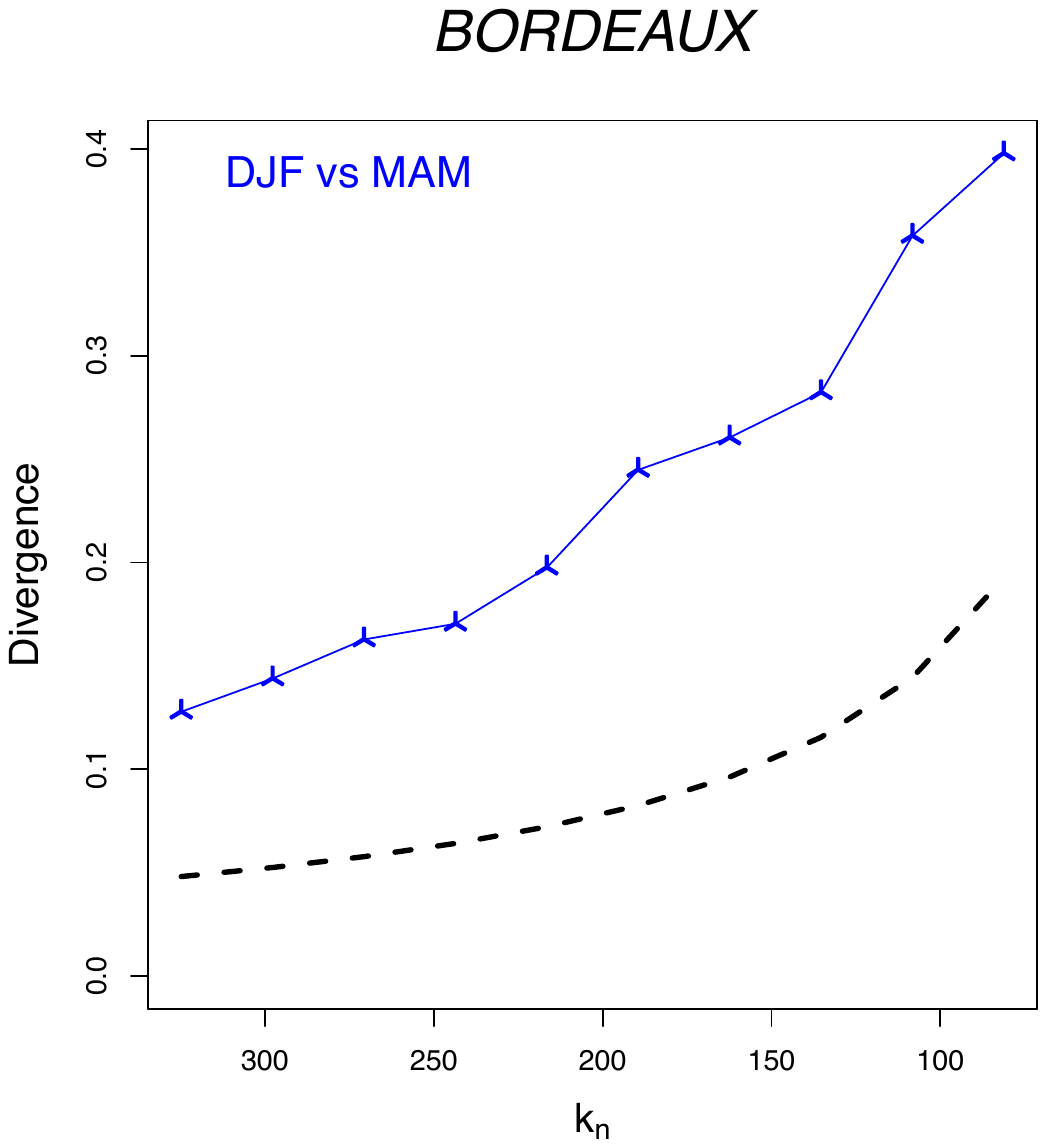}
  \hspace*{-1.8em}\includegraphics[trim={0cm 2cm 0 6cm},clip,width=0.38\linewidth]{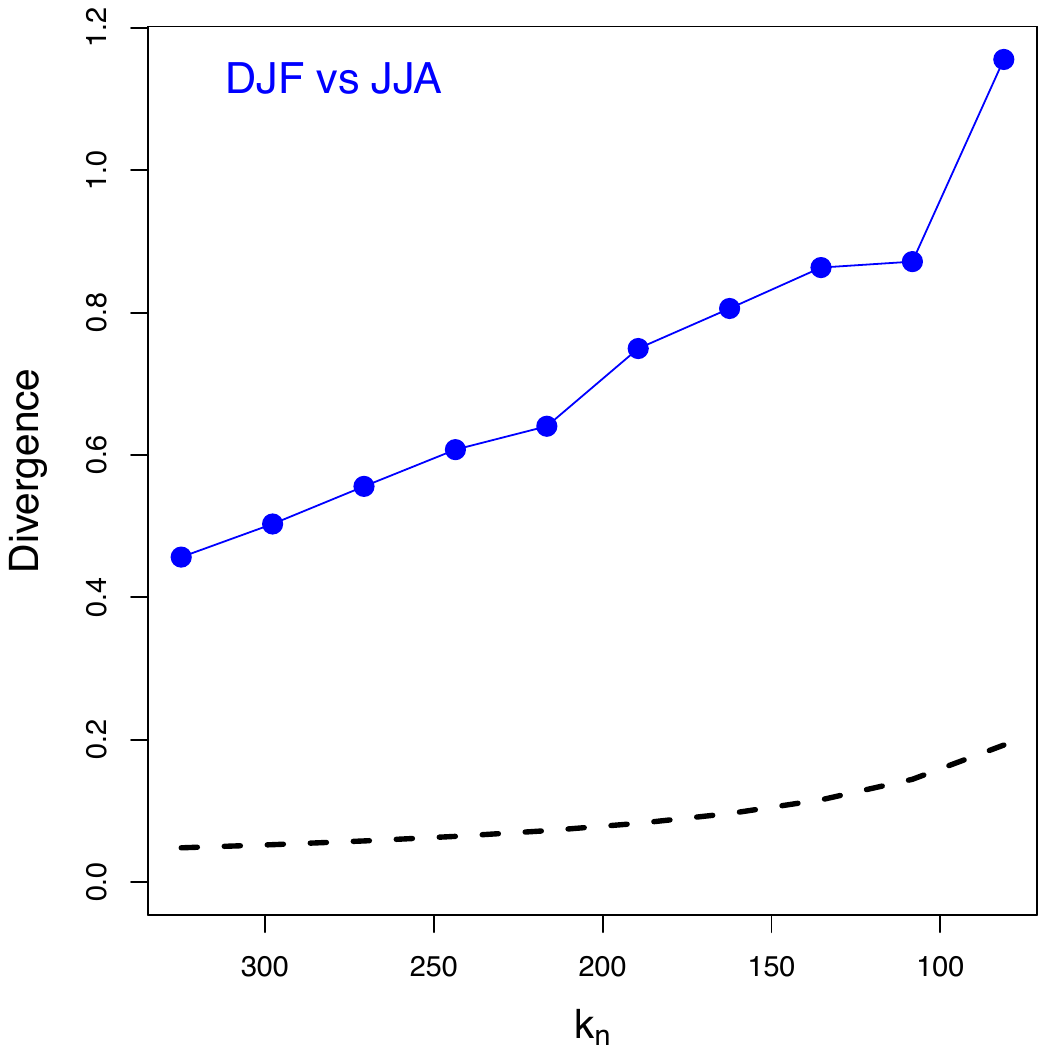}
	\hspace*{-1.8em}\includegraphics[trim={0cm 2cm 0 6cm},clip,width=0.38\linewidth]{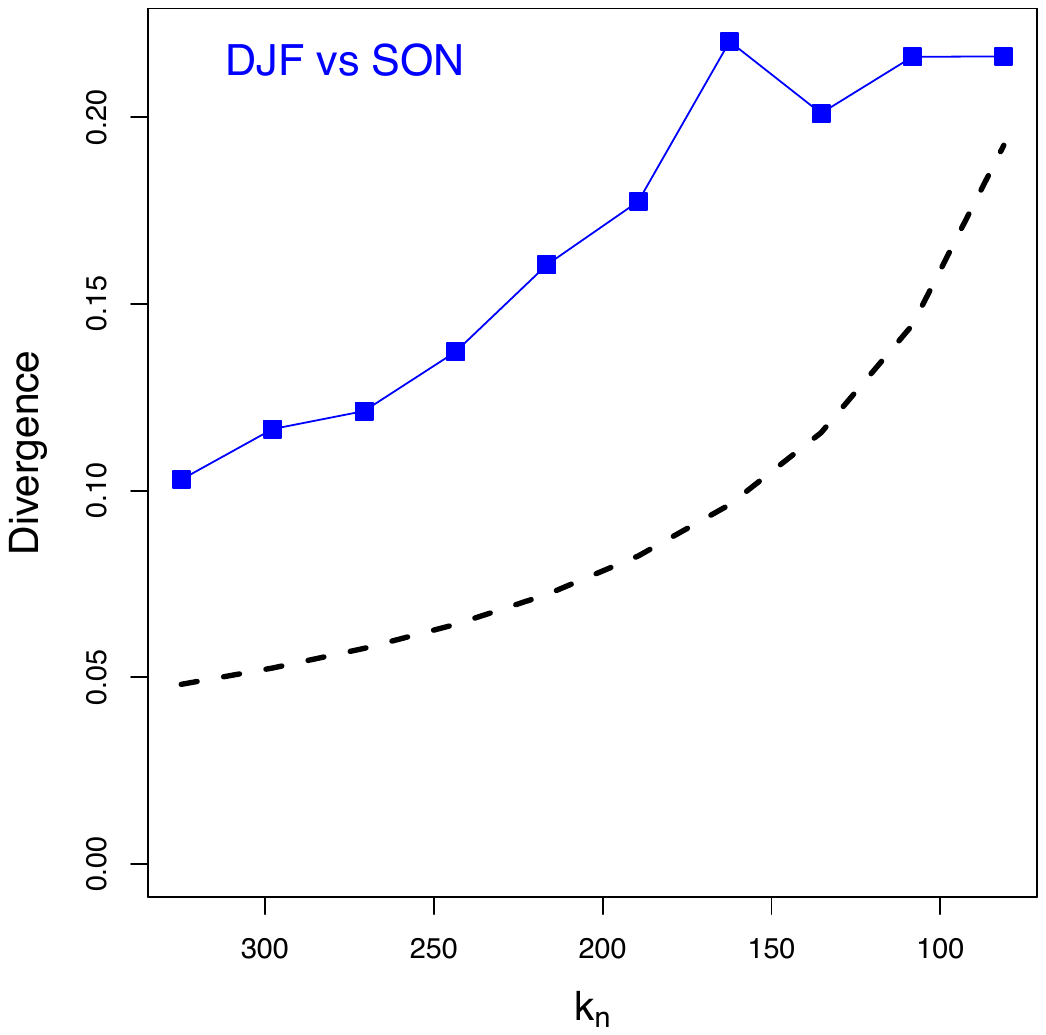}
 \hspace*{-3em}\\ \vspace*{-4.5em}%
	\hspace*{-2.7em}\includegraphics[trim={0cm 2cm 0 6cm},clip,width=0.38\linewidth]{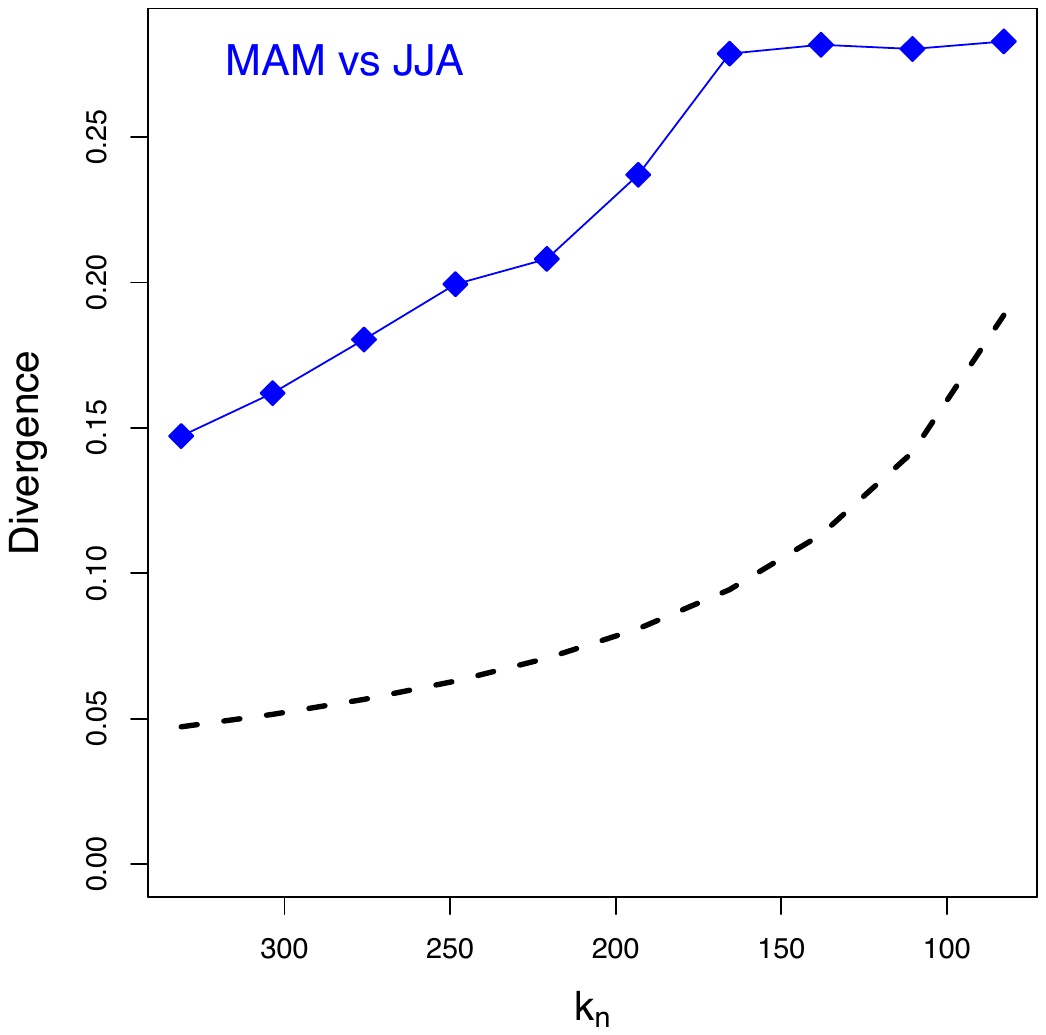}
  \hspace*{-1.8em}\includegraphics[trim={0cm 2cm 0 6cm},clip,width=0.38\linewidth]{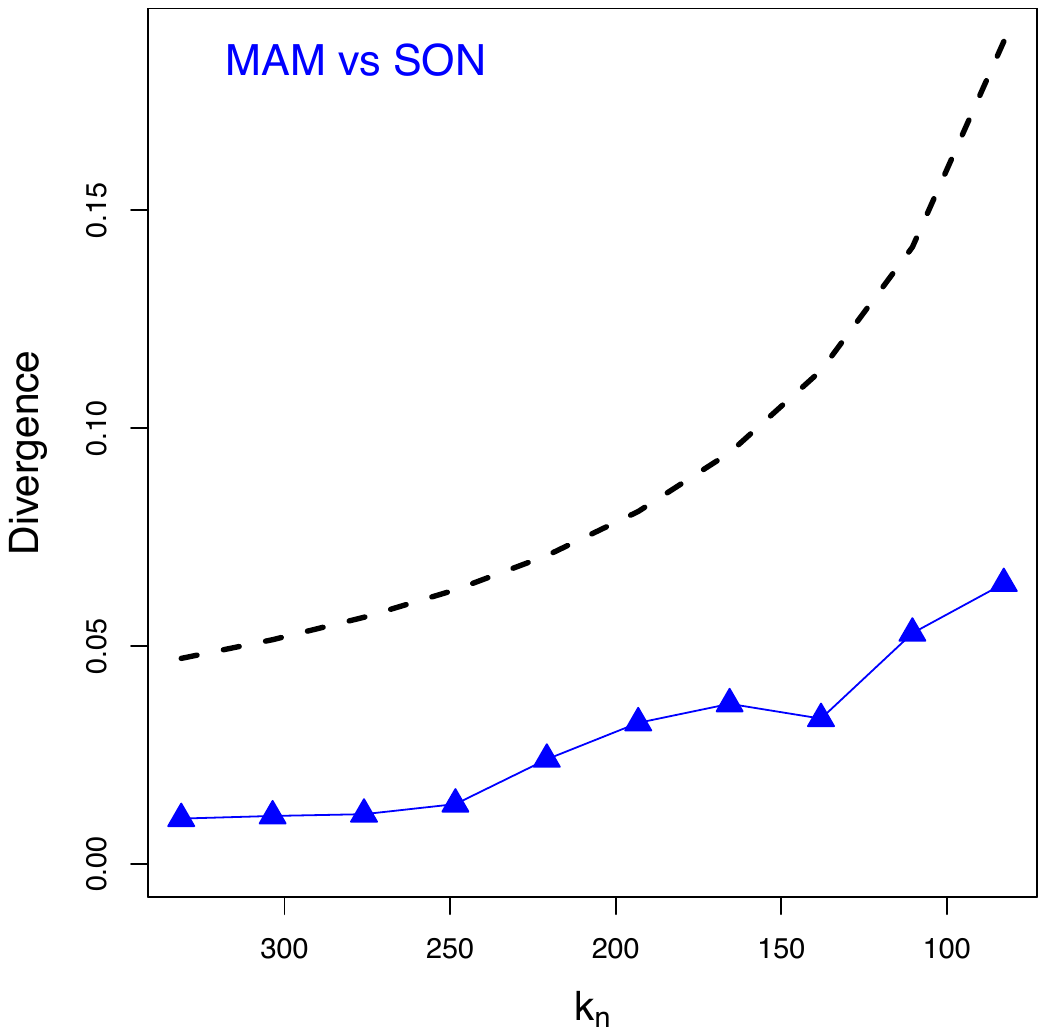}
	\hspace*{-1.8em}\includegraphics[trim={0cm 2cm 0 6cm},clip,width=0.38\linewidth]{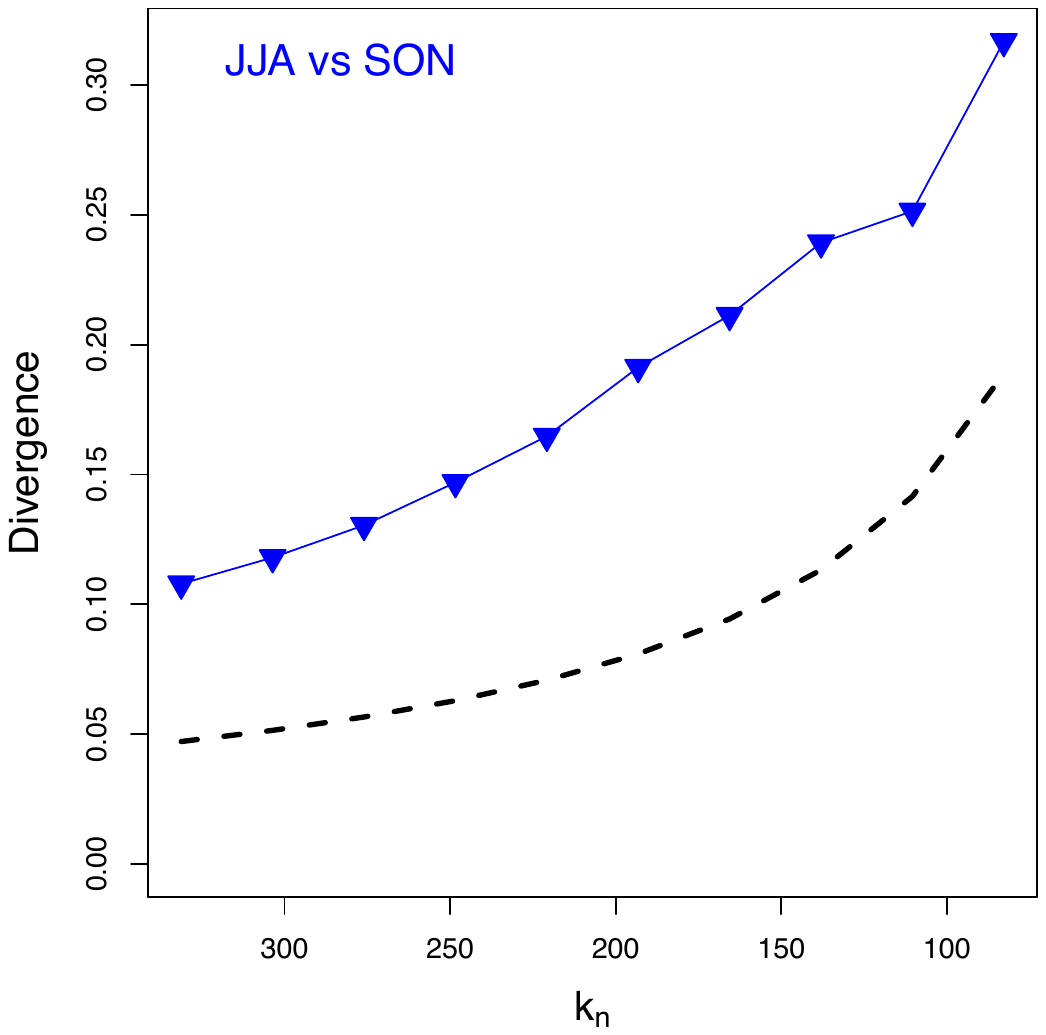}
 \vspace*{-5em}
 	\caption{For the city of Bordeaux (France), each panel compares the extremal dependence between daily  maxima of 6-min and hourly rainfall in two different seasons using the KL divergence $\hat D_K$ (blue line) based on the Euclidean risk function with $K=4$ sets, as a function of the number of exceedances $k_n$.
    Critical values at level 95\% (dashed black line) are obtained by the bootstrap procedure. 
    }
	\label{fig:bordeaux4sets}
\end{figure} 

While the extremal correlations in Figure~\ref{fig:bordeaux} seem to indicate no seasonal differences, our test rejects the null hypothesis of equal extremal dependence for all pairs of seasons
except for spring versus fall. 
The results are stable across different number of exceedances since $\hat D_K$ remains above the critical values at level 95\% (dashed lines).

A natural question is if this seasonal contrast is specific to Bordeaux, or if it generalizes in space. 
Figure \ref{fig:KL-map-4sets} indicates with orange signs whether our test rejects the null hypothesis (no seasonal change between 6-min and hourly precipitation extremes) with $K=4$ sets, $k_n=220$ exceedances and a 95\% critical level, at different locations in France.  
\begin{figure}[tb]
	\centering
  \hspace*{-1em}\includegraphics[width=.75\linewidth]{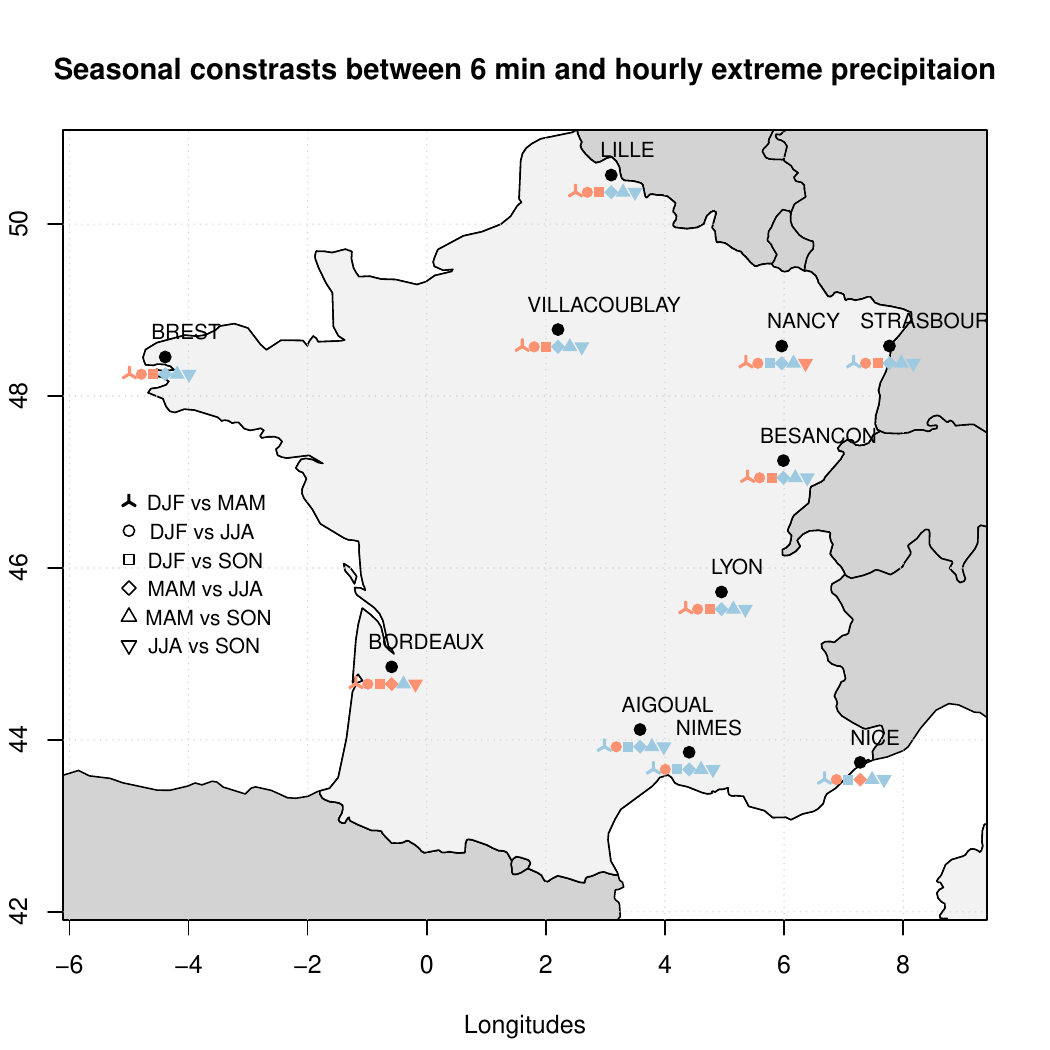}
	\caption{For each of the 11 rain gauges locations in France, an orange sign indicates that our KL divergence test, based on $\hat D_K$ with $K=4$ and $k_n=220$ exceedances, rejects the null hypothesis at the critical level 95\%. 
    }
	\label{fig:KL-map-4sets}
\end{figure} 
Generally, spring, summer and fall seem to behaves relatively similar,
whereas the winter season is clearly different. This phenomenon is particularly true for western, central and  northern stations like Brest, Lille, Villacoublay, Besancon and Bordeaux. For the other stations, the winter season mainly contrasts with spring season.
If the number $K$ of sets in the partition 
is too small, the test may lose power. For $K=2$ our test is essentially equivalent to a test based on the extremal correlation $\chi$.  
Figure \ref{fig:KL-map-2sets} shows that then differences 
cannot be identified with the test; see also the right-hand side of Figure~\ref{fig:nbsets} for a possible explanation.

From a hydrological point of view, this study emphasizes that, in addition to  well known seasonality effects on marginals at different aggregation scales, extremal dependencies between  different time scales can also be impacted by seasonal effects. 
This  result opens  research avenues in  the domain of IDF curve modeling, in particular how to physically explain and integrate this new knowledge. 

\begin{figure}[tb]
	\centering
  \hspace*{-1em}\includegraphics[width=.75\linewidth]{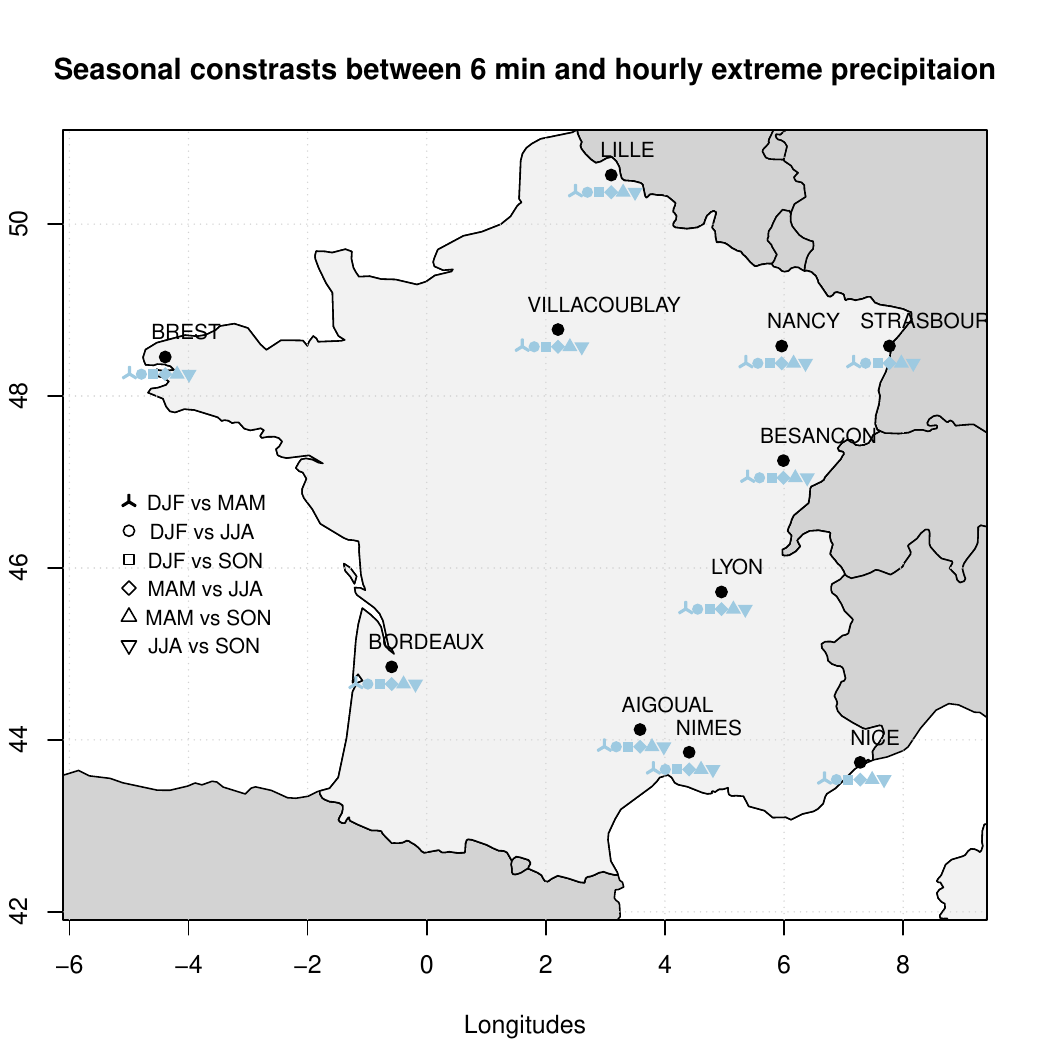}
	\caption{Same as Figure \ref{fig:KL-map-4sets} but with $K=2$ instead of $K=4$.
    }
	\label{fig:KL-map-2sets}
\end{figure}

\bibliographystyle{Chicago}

\spacingset{1.25} 
\footnotesize{\bibliography{ref}}

\begin{thebibliography}{}

\bibitem[\protect\citeauthoryear{Beirlant, Goegebeur, Teugels, and
  Segers}{Beirlant et~al.}{2005}]{beirlant_statistics_2005}
Beirlant, J., Y.~Goegebeur, J.~Teugels, and J.~Segers (2005, August).
\newblock {\em Statistics of {Extremes}: {Theory} and {Applications}}.

\bibitem[\protect\citeauthoryear{B{\"u}cher and Dette}{B{\"u}cher and
  Dette}{2013}]{bucher2013}
B{\"u}cher, A. and H.~Dette (2013).
\newblock Multiplier bootstrap of tail copulas with applications.
\newblock {\em Bernoulli\/}~{\em 19\/}(5A), 1655--1687.

\bibitem[\protect\citeauthoryear{B{\"u}cher, Kinsvater, and
  Kojadinovic}{B{\"u}cher et~al.}{2017}]{bucher2017}
B{\"u}cher, A., P.~Kinsvater, and I.~Kojadinovic (2017).
\newblock Detecting breaks in the dependence of multivariate extreme-value
  distributions.
\newblock {\em Extremes\/}~{\em 20}, 53--89.

\bibitem[\protect\citeauthoryear{Buishand, de~Haan, and Zhou}{Buishand
  et~al.}{2008}]{bui2008}
Buishand, T.~A., L.~de~Haan, and C.~Zhou (2008).
\newblock On spatial extremes: with application to a rainfall problem.
\newblock {\em Ann. Appl. Stat.\/}~{\em 2}, 624--642.

\bibitem[\protect\citeauthoryear{Coles, Heffernan, and Tawn}{Coles
  et~al.}{1999}]{col1999}
Coles, S., J.~Heffernan, and J.~Tawn (1999).
\newblock Dependence measures for extreme value analyses.
\newblock {\em Extremes\/}~{\em 2}, 339--365.

\bibitem[\protect\citeauthoryear{Coles and Tawn}{Coles and
  Tawn}{1991}]{col1991}
Coles, S.~G. and J.~A. Tawn (1991).
\newblock Modelling extreme multivariate events.
\newblock {\em J. R. Stat. Soc. Ser. B Stat. Methodol.\/}~{\em 53}, 377--392.

\bibitem[\protect\citeauthoryear{de~Fondeville and Davison}{de~Fondeville and
  Davison}{2022}]{Fondeville22}
de~Fondeville, R. and A.~C. Davison (2022, September).
\newblock {Functional peaks‐over‐threshold analysis}.
\newblock {\em Journal of the Royal Statistical Society Series B\/}~{\em
  84\/}(4), 1392--1422.

\bibitem[\protect\citeauthoryear{de~Haan and Ferreira}{de~Haan and
  Ferreira}{2006}]{deh2006a}
de~Haan, L. and A.~Ferreira (2006).
\newblock {\em Extreme Value Theory}.
\newblock New York: Springer.

\bibitem[\protect\citeauthoryear{de~Haan and Pereira}{de~Haan and
  Pereira}{2006}]{deh2006}
de~Haan, L. and T.~T. Pereira (2006).
\newblock Spatial extremes: models for the stationary case.
\newblock {\em Ann. Statist.\/}~{\em 34}, 146--168.

\bibitem[\protect\citeauthoryear{Dombry, Engelke, and Oesting}{Dombry
  et~al.}{2016}]{dom2016}
Dombry, C., S.~Engelke, and M.~Oesting (2016).
\newblock Exact simulation of max-stable processes.
\newblock {\em Biometrika\/}~{\em 103}, 303--317.

\bibitem[\protect\citeauthoryear{Dombry and Ribatet}{Dombry and
  Ribatet}{2015}]{dom2015}
Dombry, C. and M.~Ribatet (2015).
\newblock Functional regular variations, {P}areto processes and peaks over
  threshold.
\newblock {\em Statistics and its Interface\/}~{\em 8}, 9--17.

\bibitem[\protect\citeauthoryear{Dudley}{Dudley}{1978}]{Dudley1978}
Dudley, R. (1978).
\newblock Central limit theorems for empirical measures.
\newblock {\em Annals of Probability\/}~{\em 6\/}(6), 899--929.

\bibitem[\protect\citeauthoryear{Einmahl, de~Haan, and Sinha}{Einmahl
  et~al.}{1997}]{einmahl1997}
Einmahl, H., L.~de~Haan, and A.~Sinha (1997).
\newblock Estimating the spectral measure of an extreme value distribution.
\newblock {\em Stochastic Processes and their Applications\/}~{\em 70\/}(2),
  143--171.

\bibitem[\protect\citeauthoryear{Einmahl, Piterbarg, and de~Haan}{Einmahl
  et~al.}{2001}]{einmahl2001}
Einmahl, J.~H., V.~I. Piterbarg, and L.~de~Haan (2001).
\newblock Nonparametric estimation of the spectral measure of an extreme value
  distribution.
\newblock {\em Annals of Statistics\/}~{\em 29\/}(5), 1401--1423.

\bibitem[\protect\citeauthoryear{Embrechts, Kl\"{u}ppelberg, and
  Mikosch}{Embrechts et~al.}{1997}]{emb1997}
Embrechts, P., C.~Kl\"{u}ppelberg, and T.~Mikosch (1997).
\newblock {\em Modelling Extremal Events: for Insurance and Finance}.
\newblock London: Springer.

\bibitem[\protect\citeauthoryear{Engelke, Hitz, Gnecco, and Hentschel}{Engelke
  et~al.}{2024}]{graphicalExtremes2024}
Engelke, S., S.~A. Hitz, N.~Gnecco, and M.~Hentschel (2024).
\newblock {\em graphical{E}xtremes: Statistical Methodology for Graphical
  Extreme Value Models}.
\newblock Available from
  \url{https://CRAN.R-project.org/package=graphicalExtremes}, R package version
  0.3.2.

\bibitem[\protect\citeauthoryear{Engelke, Malinowski, Kabluchko, and
  Schlather}{Engelke et~al.}{2015}]{eng2014}
Engelke, S., A.~Malinowski, Z.~Kabluchko, and M.~Schlather (2015).
\newblock Estimation of {H}\"usler--{R}eiss distributions and
  {B}rown--{R}esnick processes.
\newblock {\em J. R. Stat. Soc. Ser. B Stat. Methodol.\/}~{\em 77}, 239--265.

\bibitem[\protect\citeauthoryear{Gudendorf and Segers}{Gudendorf and
  Segers}{2010}]{seg2010}
Gudendorf, G. and J.~Segers (2010).
\newblock Extreme-value copulas.
\newblock In {\em Copula Theory and Its Applications}, pp.\  127--145.
  Springer.

\bibitem[\protect\citeauthoryear{Gumbel}{Gumbel}{1960}]{gum1960}
Gumbel, E.~J. (1960).
\newblock Distributions de valeurs extr{\^{e}}mes en plusieurs dimensions.
\newblock {\em Publ. Inst. Statist. Paris\/}~{\em 9}, 171--173.

\bibitem[\protect\citeauthoryear{Haruna, Blanchet, and Favre}{Haruna
  et~al.}{2023}]{haruna2023modeling}
Haruna, A., J.~Blanchet, and A.-C. Favre (2023).
\newblock Modeling intensity-duration-frequency curves for the whole range of
  non-zero precipitation: a comparison of models.
\newblock {\em Water Resources Research\/}~{\em 59\/}(6), e2022WR033362.

\bibitem[\protect\citeauthoryear{Hershfield}{Hershfield}{1961}]{Hershfield61}
Hershfield, D.~M. (1961).
\newblock Rainfall frequency atlas of the united states for durations from 30
  minutes to 24 hours and return periods from 1-100 years.
\newblock Technical report, Weather Bureau, Department of Commerce (Washington,
  DC).

\bibitem[\protect\citeauthoryear{Hofert, Kojadinovic, Maechler, and Yan}{Hofert
  et~al.}{2023}]{copula2023}
Hofert, M., I.~Kojadinovic, M.~Maechler, and J.~Yan (2023).
\newblock {\em copula: Multivariate Dependence with Copulas}.
\newblock R package version 1.1-3.

\bibitem[\protect\citeauthoryear{Huser and Davison}{Huser and
  Davison}{2014}]{hus2013}
Huser, R. and A.~C. Davison (2014).
\newblock Space--time modelling of extreme events.
\newblock {\em J. R. Stat. Soc. Ser. B Stat. Methodol.\/}~{\em 76}, 439--461.

\bibitem[\protect\citeauthoryear{Li, Gao, Yin, and Wu}{Li
  et~al.}{2024}]{li2024}
Li, Y., J.~Gao, J.~Yin, and S.~Wu (2024).
\newblock Assessing the potential of compound extreme storm surge and
  precipitation along china's coastline.
\newblock {\em Weather and Climate Extremes\/}~{\em 45}, 100702.

\bibitem[\protect\citeauthoryear{Nasr, Wahl, Rashid, Camus, and Haigh}{Nasr
  et~al.}{2021}]{nas2021}
Nasr, A.~A., T.~Wahl, M.~M. Rashid, P.~Camus, and I.~D. Haigh (2021).
\newblock Assessing the dependence structure between oceanographic, fluvial,
  and pluvial flooding drivers along the united states coastline.
\newblock {\em Hydrology and Earth System Sciences\/}~{\em 25\/}(12),
  6203--6222.

\bibitem[\protect\citeauthoryear{Nasr, Wahl, Rashid, Jane, Camus, and
  Haigh}{Nasr et~al.}{2023}]{nas2023}
Nasr, A.~A., T.~Wahl, M.~M. Rashid, R.~A. Jane, P.~Camus, and I.~D. Haigh
  (2023).
\newblock Temporal changes in dependence between compound coastal and inland
  flooding drivers around the contiguous united states coastline.
\newblock {\em Weather and Climate Extremes\/}~{\em 41}, 100594.

\bibitem[\protect\citeauthoryear{Naveau, Guillou, and Rietsch}{Naveau
  et~al.}{2014}]{nav2014}
Naveau, P., A.~Guillou, and T.~Rietsch (2014).
\newblock A non-parametric entropy-based approach to detect changes in climate
  extremes.
\newblock {\em J. R. Stat. Soc. Ser. B Stat. Methodol.\/}~{\em 76}, 861--884.

\bibitem[\protect\citeauthoryear{Naveau, Hannart, and Ribes}{Naveau
  et~al.}{2020}]{naveau_statistical_2020}
Naveau, P., A.~Hannart, and A.~Ribes (2020).
\newblock Statistical {Methods} for {Extreme} {Event} {Attribution} in
  {Climate} {Science}.
\newblock {\em Annual Reviews of Statistics and its Application\/}.
\newblock Publisher: Annual Reviews.

\bibitem[\protect\citeauthoryear{Resnick}{Resnick}{2008}]{res2008}
Resnick, S.~I. (2008).
\newblock {\em Extreme Values, Regular Variation and Point Processes}.
\newblock New York: Springer.

\bibitem[\protect\citeauthoryear{Schlather and Tawn}{Schlather and
  Tawn}{2003}]{sch2003}
Schlather, M. and J.~A. Tawn (2003).
\newblock A dependence measure for multivariate and spatial extreme values:
  Properties and inference.
\newblock {\em Biometrika\/}~{\em 90}, 139--156.

\bibitem[\protect\citeauthoryear{Ulrich, Fauer, and Rust}{Ulrich
  et~al.}{2021}]{ulrich2021modeling}
Ulrich, J., F.~S. Fauer, and H.~W. Rust (2021).
\newblock Modeling seasonal variations of extreme rainfall on different
  timescales in germany.
\newblock {\em Hydrology and Earth System Sciences\/}~{\em 25\/}(12),
  6133--6149.

\bibitem[\protect\citeauthoryear{Ulrich, Jurado, Peter, Scheibel, and
  Rust}{Ulrich et~al.}{2020}]{Ulrich20}
Ulrich, J., O.~E. Jurado, M.~Peter, M.~Scheibel, and H.~W. Rust (2020).
\newblock Estimating idf curves consistently over durations with spatial
  covariates.
\newblock {\em Water\/}~{\em 12\/}(11:3119).

\bibitem[\protect\citeauthoryear{Vignotto, Engelke, and Zscheischler}{Vignotto
  et~al.}{2021}]{vig2021}
Vignotto, E., S.~Engelke, and J.~Zscheischler (2021).
\newblock Clustering bivariate dependencies of compound precipitation and wind
  extremes over great britain and ireland.
\newblock {\em Weather and Climate Extremes\/}~{\em 32}, 100318.

\bibitem[\protect\citeauthoryear{Zscheischler, Martius, Westra, Bevacqua,
  Raymond, Horton, Hurk, AghaKouchak, Jézéquel, Mahecha, Maraun, Ramos,
  Ridder, Thiery, and Vignotto}{Zscheischler
  et~al.}{2020}]{Zscheischler20review}
Zscheischler, J., O.~Martius, S.~Westra, E.~Bevacqua, C.~Raymond, R.~Horton,
  B.~Hurk, A.~AghaKouchak, A.~Jézéquel, M.~Mahecha, D.~Maraun, A.~Ramos,
  N.~Ridder, W.~Thiery, and E.~Vignotto (2020, 06).
\newblock A typology of compound weather and climate events.
\newblock {\em Nature Reviews Earth {\&} Environment\/}~{\em 1}, 1--15.

\bibitem[\protect\citeauthoryear{Zscheischler, Naveau, Martius, Engelke, and
  Raible}{Zscheischler et~al.}{2020}]{Zscheischler20}
Zscheischler, J., P.~Naveau, O.~Martius, S.~Engelke, and C.~C. Raible (2020).
\newblock Evaluating the dependence structure of compound precipitation and
  wind speed extremes.
\newblock {\em Earth System Dynamics Discussions\/}~{\em 2020}, 1--23.

\bibitem[\protect\citeauthoryear{Zscheischler and Seneviratne}{Zscheischler and
  Seneviratne}{2017}]{zsc2017}
Zscheischler, J. and S.~I. Seneviratne (2017).
\newblock Dependence of drivers affects risks associated with compound events.
\newblock {\em Science Advances\/}~{\em 3}.

\end{thebibliography}

\newpage

\setcounter{page}{1}

\appendix

\spacingset{1.8}

\begin{center}
 {\Large \textbf{SUPPLEMENTARY MATERIAL TO \\
\vspace{12pt}
``A Kullback--Leibler divergence test for multivariate extremes: theory and practice'' }}
\end{center}

\section{Proofs}\label{sec: Proof}
\begin{proof}[Proof of Proposition \ref{prop: values of p}] Choose $0<c<\min\set{\norm{\g x}_\infty: r(\g x)>1}$. Then $\Omega_r\subset \set{\g x :\norm{\g x}_\infty> c}$. Note that
    $uB\subset u\Omega_r \subset\set{\g x :\norm{\g x}_\infty> uc}$. We have that
    \begin{align*}
        P(\g X\in uB|r(\g X)>u)&=\frac{P(\g x \in uB)}{P(\g x \in u\Omega_r)}\\
        &=\frac{P(\g x \in uB, \norm{\g x}_\infty> uc)}{P(\g x \in u\Omega_r, \norm{\g x}_\infty> uc)}\\
        &=\frac{P(\g x \in uB| \norm{\g x}_\infty> uc)}{P(\g x \in u\Omega_r| \norm{\g x}_\infty> uc)}
    \end{align*}
    By regarding $uc$ as $u$ in \eqref{mevd} and take $u\to\infty$, we get that
    $$\lim_{u\to\infty} P(\g X\in uB|r(\g X)>u)=\frac{\nu(B/c)}{\nu(\Omega_r/c)}=\frac{\nu(B)}{\nu(\Omega_r)},$$
    where the last line follows from the homogeneity property of $\nu$.
\end{proof}

To prove Theorem \ref{thm_AD_hom}, we first establish the joint asymptotic behavior of the estimators $\hat p_j$ in the following Proposition.
\begin{prop} \label{prop: probability convergence estimate threshold}
Assume the same conditions as in Theorem \ref{thm_AD_hom}. Then, we have that for all $1\leq j\leq K$ jointly, as $n\to\infty$,
$$\sqrt{k_n} \left(\hat p_{j}-p_{j}\right)   \td  N_j-p_j\sum_{j=1}^K N_j,$$
where $N_1,N_2,\ldots,N_K$ are independent normally distributed random variables with mean zero and variance $p_{1}, p_{2},\ldots,p_K$ respectively.
\end{prop}

Consider sets $S(t,j)=tA_j$ for $t\in[a,b]$ with any fixed $0<a<b<+\infty$, $1\leq j\leq K$. Denote $\tilde u_n=U^{-1}_X(k_n/n)$. To prove Proposition \ref{prop: probability convergence estimate threshold}, we intend to  apply Lemma 3.1 in \cite{einmahl1997} to the sequence $\{\g Z_i^{(n)}\}_{i=1}^n=\{\g X_i/\tilde u_n\}_{i=1}^n$ and the class of sets $\{S(t,j):t\in[a,b], 1\leq j\leq K\}$. To check the conditions therein, we first show the following preliminary lemma.

\begin{lemma} \label{lemma: VC class}
  Assume that Assumption \ref{ass:sets} holds. Then the class of sets $\{S(t,j):t\in[a,b], 1\leq j\leq K\}$ forms a VC class with finite dimension. In addition, they satisfy the condition (SE) in Gaenssler (1983, p.108)
\end{lemma}
\begin{proof}[Proof of Lemma \ref{lemma: VC class}]
Following Assumption \ref{ass:sets}, there must exist finite number of Borel sets in $\mathcal A_r$, denoted as $\{B_1,\ldots, B_L\}$ such that for all $j=1,2,\ldots, K$, $A_j$ can be expressed using a finite number of the operations of union, intersection, and complementation on the sets in $\{tB_l:t\in[a, b], 1\leq l\leq L\}$. Without loss of generality, we assume that $B_l\subset \Omega_r$ for all $l=1,2,\ldots, L$.

Firstly, we show that for each given $l$ the class of sets $\set{tB_l: t\in[a,b]}$ is a VC class with dimension 3. For any given 3 points in $a\Omega_r$, If there are at least two falling outside $aB_l$, their subset cannot be shattered by the class. If there is at most one falling outside $aB_l$, i.e. at least two points $x_1$ and $x_2$ are within the set $aB_l$. Since $tB_l\subset sB_l$ for any $t>s$, for each point $x_i$ there exists a unique value $t_i$ such that $x_i\in tB_l$ for all $t<t_i$ and $x_i\in (tB_l)^c$ for all $t>t_i$. By consider all three situations $t_1<t_2$, $t_1=t_2$ and $t_1>t_2$, it is straightforward to obtain that in each case, it is not possible to shatter all the subsets of the two points $x_1$ and $x_2$ using $tB_l$ for $t\in[a,b]$. 
	
Secondly, since combining finite number of VC classes with finite dimension will lead to a VC class with finite dimension, we get that the class of sets $\{tB_l:t\in[a,b], 1\leq l\leq L\}$ forms a VC class with finite dimension

Lastly, since any set in $S(t,j)$ with $t\in[a, b], 1\leq j\leq K$ can be expressed using a finite number of the operations of union, intersection, and complementation on the sets in $\{tB_l:t\in[a, b], 1\leq l\leq L\}$, the class $\{S(t,j):t\in[a, b], 1\leq j\leq K\}$ also forms a VC class with finite dimension; see Proposition 7.12 in \cite{Dudley1978}. 
	
The second half of the lemma regarding the condition (SE) is obvious by considering all rational numbers $t\in[a, b]$.
\end{proof}

\begin{proof}[Proof of Proposition \ref{prop: probability convergence estimate threshold}]

Recall that $\tilde u_n=U_X^{-1}(k_n/n)$, which implies $P(\norm{\g X}_\infty >\tilde u_n)=\frac{k_n}{n}$ and $\tilde u_n\to \infty$ as $n\to\infty$. Further recall that $\{\g Z_i^{(n)}\}_{i=1}^n=\{\g X_i/\tilde u_n\}_{i=1}^n$ and $S(t,j)=tA_j$. By applying the second order condition \eqref{mevd_2nd} to $A_j$, we obtain that
as $n\to\infty$,
\begin{equation}\label{limit for A}
\frac{n}{k_n}\mathbb P(\g Z_i^{(n)}\in S(t,j))=\frac{\mathbb P(\g X_i\in \tilde u_n t A_j)}{P(\norm{\g X}_\infty >\tilde u_n)}\to \nu_X(tA_j)= t^{-1}\nu_X(A_j),
\end{equation}
uniformly for all $t\in[a, b]$ and $j=1,2,\ldots,K$. This verifies condition (3.1) in Lemma 3.1 in \cite{einmahl1997}  with a limit measure $\mu=\nu_X$. In fact, we have a more precise limit relation: as $n\to\infty$,
\begin{equation}\label{limit for A 2nd}
  \frac{n}{k_n}\mathbb P(\g Z_i^{(n)}\in S(t,j))-t^{-1}\nu_X(A_j)=O(a_X(\tilde u_n)),
\end{equation}
uniformly for all $t\in[a, b]$ and $j=1,2,\ldots,K$.

With the other conditions checked in Lemma \ref{lemma: VC class} above, we can apply Lemma 3.1 in \cite{einmahl1997} to the sequence $\{\g Z_i^{(n)}\}_{i=1}^n=\{\g X_i/\tilde u_n\}_{i=1}^n$ and the class of sets $\mathcal{B}:=\{S(t,j):t\in[a, b], 1\leq j\leq K\}$ and obtain the following result. Under a proper Skorokhod construction, there exists a series of bounded, uniformly continuous, mean zero Gaussian process $W_n$ defined on $S\in \mathcal{B}$ such that for any $S_1, S_2\in  \mathcal{B}$
$$\text{Cov}(W_n(S_1),W_n(S_2))=\mu(S_1\bigcap S_2),$$
and as $n\to\infty$,
$$\abs{\sqrt{k_n}\suit{\frac{1}{k_n}\sum_{i=1}^n1_{\set{\g X_i\in t\tilde u_n A_j}}-\frac{n}{k_n}\mathbb P(\g X_i\in t\tilde u_n A_j)}- W_n(tA_j)}\stackrel{a.s.}{\to} 0,$$
uniformly for $t\in[a, b]$ and $j=1,2,\ldots,l$.

From \eqref{limit for A 2nd} and the condition that $\sqrt{k_n}a_X(\tilde u_n)\to 0$ as $n\to\infty$, we can replace the deterministic term $\frac{n}{k_n}\mathbb P(\g X_i\in t\tilde u_n A_j)$ by its limit, and obtain that
as $n\to\infty$,
$$
\abs{\sqrt{k_n}\suit{\frac{1}{k_n}\sum_{i=1}^n1_{\set{\g X_i\in t\tilde u_n A_j}}-t^{-1}\nu_X(A_j)}- W_n(tA_j)}\stackrel{a.s.}{\to} 0,
$$
uniformly for $t\in[a, b]$ and $j=1,2,\ldots,l$. Divide both sides with $\nu_X(\Omega_r)$, together with defining a series fo Gaussian process as $\tilde W_n=W_n/\nu_X(\Omega_r)$, we get that 
\begin{equation} \label{eq: process result for convergence}
  \abs{\sqrt{k_n}\suit{\frac{1}{k_n\nu_X(\Omega_r)}\sum_{i=1}^n1_{\set{\g X_i\in t\tilde u_n A_j}}-t^{-1}p_j}- \tilde W_n(tA_j)}\stackrel{a.s.}{\to} 0,
  \end{equation}

We intend to change $t$ in \eqref{eq: process result for convergence} by $t_n=u_n/\tilde u_n=R^{(X)}_{n-k_n,n}/\tilde u_n$. For that reason, we first have to derive the asymptotic behavior for $t_n$. By adding the limit relation \eqref{eq: process result for convergence} for all $1\leq j\leq K$ and note that $\set{A_j}_{j=1}^K$ forms a partition of $\Omega_r$ we get that as $n\to\infty$, for $t\in[a,b]$
$$
\sqrt{k_n}\suit{\frac{1}{k_n\nu_X(\Omega_r)}\sum_{i=1}^n1_{\set{r(\g X_i)>t\tilde u_n}}-t^{-1}}- \tilde W_n(t\Omega_r)\to 0 \mbox{\ \ a.s.}
$$

By applying the Vervaat's lemma to this limit relation, we get that as $n\to\infty$, for $v\in [b^{-1},a^{-1}]$,
$$
\sqrt{k_n}\suit{\frac{R^{(X)}_{n-[k_n\nu_X(\Omega_r)v],n}}{\tilde u_n}-v^{-1}}-v^{-2}\tilde W_n\suit{v^{-1}\Omega_r}\to 0  \mbox{\ \ a.s.}
$$
With suitable choice of $a$ and $b$, such that $1/\nu_X(\Omega_r)\in[b^{-1},a^{-1}]$, we can take $v=1/\nu_X(\Omega_r)$, and get that as $n\to\infty$,
$$
\sqrt{k_n}\suit{\frac{R^{(X)}_{n-k_n,n}}{\tilde u_n}-\nu_X(\Omega_r)}- (\nu_X(\Omega_r))^{2}\tilde W_n(\nu_X(\Omega_r)\cdot\Omega_r)\to 0  \mbox{\ \ a.s.}
$$
Since $\nu_X(\Omega_r)\in[a,b]$, we can conduct the planned replacement. With replacing $t$ by $t_n=u_n/\tilde u_n$ in \eqref{eq: process result for convergence}, we get that as $n\to\infty$,
$$\sqrt{k_n}\suit{\frac{1}{\nu_X(\Omega_r)}\hat p_j-\suit{\frac{u_n}{\tilde u_n}}^{-1}p_j}-\tilde W_n(\nu_X(\Omega_r)\cdot A_j)\to 0 \mbox{\ \ a.s.}$$
Together with applying the Delta method to the asymptotic behavior of $u_n/\tilde u_n$, we get that, as $n\to\infty$, 
$$\sqrt{k_n}\suit{\frac{\hat p_j}{\nu_X(\Omega_r)}-\frac{p_j}{\nu_X(\Omega_r)}}-\suit{\tilde W_n(\nu_X(\Omega_r)\cdot A_j)-p_j\tilde W_n(\nu_X(\Omega_r)\cdot \Omega_r)}\to 0 \mbox{\ \ a.s.}$$
for all $j=1,2,\ldots,K$.
Recall the notation $\tilde W_n=W_n/\nu_X(\Omega_r)$. We thus conclude that as $n\to\infty$
\begin{equation} \label{almost final result for p}
  \sqrt{k_n}(\hat p_j-p_j)\stackrel{a.s.}{\to}W_n(\nu_X(\Omega_r)\cdot A_j)-p_jW_n(\nu_X(\Omega_r)\cdot \Omega_r),
\end{equation}
for all $j=1,2,\ldots,K$. Denote $N_j:=W_n(\nu_X(\Omega_r)\cdot A_j)$. Clearly $N_j$ are independently normally distributed random variables with mean zero and $W_n(\nu_X(\Omega_r)\cdot \Omega_r)=\sum_{j=1}^K N_j$.

The remaining part of the proof is to calculate the variance of each $N_j$, for $j=1,2,\ldots, K$:
$$\Var(N_j)=\nu_X(\nu_X(\Omega_r)\cdot A_j)=\frac{\nu_X(A_j)}{\nu_X(\Omega_r)}=p_j.$$
The proposition is thus proved.
\end{proof}

Now we can turn to prove the main theorems.
\begin{proof}[Proof of Theorem \ref{thm_AD_hom} (and Theorem \ref{thm_generic})]
Proposition \ref{prop: probability convergence estimate threshold} shows that as $n\to\infty$,
$$\sqrt{k_n} \left(\hat p_{j}-p_{j}\right)   \td  N^X_j:=N_j-p_j\sum_{j=1}^K N_j,$$
where $N^X_j\sim N(0,p_j(1-p_j))$ and
$$\Cov (N^X_{j_1}, N^X_{j_2})=-p_{j_1}p_{j_2},$$
for $1\leq j_1<j_2\leq K$.

Similarly, based on the observations $\g Y_1,\g Y_2,\ldots, \g Y_n$, we construct the conditional probability estimators $\hat q_j$. They have the following asymptotic property: as $n\to\infty$, for all $1\leq j\leq K$ jointly
$$\sqrt{k_n} \left(\hat q_{j}-q_{j}\right)   \td  N^Y_j,$$
where $N^Y_j\sim N(0,q_j(1-q_j))$ and
$$\Cov (N^Y_{j_1}, N^Y_{j_2})=-q_{j_1}q_{j_2},$$
for $1\leq j_1<j_2\leq K$. Since the observations drawn from $\g X$ and $\g Y$ are independent, we get that $(N^X_1,N^X_2,\ldots,N^X_K)$ and $(N^Y_1,N^Y_2,\ldots,N^Y_K)$ are independent.

Recall the multinomial \KL{} divergence defined in \eqref{eq:KL_Multinomial}. We derive the partial derivatives of its population version, $D_K$, with respect to each dimension as follows:
\begin{align*}
D_{1,j}&=\frac{\partial D_K}{\partial p_j}=\log \frac{p_j}{q_j}+1-\frac{q_j}{p_j},\\
D_{2,j}&=\frac{\partial D_K}{\partial q_j}=\log \frac{q_j}{p_j}+1-\frac{p_j}{q_j},\\
D_{11,j}&=\frac{\partial^2 D_K}{\partial p_{j}^2}=\frac{1}{p_j}+\frac{q_j}{p_j^2},\\
D_{22,j}&=\frac{\partial^2 D_K}{\partial q_{j}^2}=\frac{1}{q_j}+\frac{p_j}{q_j^2},\\
D_{12,j}&=\frac{\partial^2 D_K}{\partial p_{j}\partial q_j}=-\frac{1}{p_j}-\frac{1}{q_j}.
\end{align*}
Notice that except those listed above, all other second order partial derivatives are zero.

By using a Taylor expansion of $D_K$ we can get that as $n\to\infty$,
\begin{align*}
\hat D_K& =  D_K+\sum_{j=1}^K\suit{D_{1,j}(\hat p_j-p_j)+D_{2,j}(\hat q_j-q_j)}\\
&\:\:\:\:\:\:+\frac{1}{2}\sum_{j=1}^K\suit{D_{11,j}(\hat p_j-p_j)^2+D_{22,j}(\hat q_j-q_j)^2+2D_{12,j}(\hat p_j-p_j)(\hat q_j-q_j)}\\
&\:\:\:\:\:\:+o_P\suit{\sum_{j=1}^K\suit{(\hat p_{j}- p_{j})^2+(\hat q_{j}- q_{j})^2}}.
\end{align*}

If $p_j=q_j$ for all $1\leq j\leq K$, we have that $D_K=0, D_{1,j}=D_{2,j}=0$ for all $1\leq j\leq K$. Therefore, only the second order terms remain. Further $D_{11,j}=D_{22,j}=\frac{2}{p_j}=\frac{2}{q_j}$ and $D_{12,j}=-\frac{2}{p_j}=-\frac{2}{q_j}$.
Therefore, as $n\to\infty$,
\begin{align*}
\frac{k_n}{2}\hat D_K& = \frac{k_n}{2}\sum_{j=1}^K\frac{1}{p_j}\suit{(\hat p_j-p_j)^2+(\hat q_j-q_j)^2-2(\hat p_j-p_j)(\hat q_j-q_j)}+o_P(1)\\
&=\frac{k_n}{2}\sum_{j=1}^K\frac{1}{p_j}(\hat p_j-\hat q_j)^2+o_P(1)\\
&\td \sum_{j=1}^K \frac{1}{2p_j}(N^X_j-N^Y_j)^2
\end{align*}
Note that this is exactly the generic result in Theorem \ref{thm_generic} under the null hypothesis.

Under the setup of Theorem \ref{thm_AD_hom}, we can further derive the limit distribution as follows. Denote $\mathbf{L}=(L_1,L_2,\ldots,L_{K-1})^T$ with $L_j=N^X_j-N^Y_j$ for $j=1,2,\ldots,K$. Then we have that $\mathbf{L}\sim N(0,\Sigma)$, where $\Sigma=(\sigma_{j_1,j_2})_{1\leq j_1,j_2\leq K-1}$, with $\sigma_{j,j}=2p_j(1-p_j)$ and $\sigma_{j_1,j_2}=-2p_{j_1}p_{j_2}$ for $j_1\neq j_2$. Notice that $\sum_{j=1}^K L_j=0$, we can rewrite
$$\sum_{j=1}^K \frac{1}{2p_j}(N^X_j-N^Y_j)^2=\sum_{j=1}^{K-1} \frac{1}{2p_j}L_j^2+\frac{1}{2 p_K}\suit{\sum_{j=1}^{K-1} L_j}^2=\mathbf{L}^T\Sigma^{-1}\mathbf{L}\sim \chi^2(K-1).$$
Here in the last step, we use the fact that $\Sigma^{-1}=(\tilde \sigma_{j_1,j_2})_{1\leq j_1,j_2\leq K-1}$, with $\tilde \sigma_{j,j}=\frac{1}{2}\suit{\frac{1}{p_j}+\frac{1}{p_K}}$ and $\tilde \sigma_{j_1,j_2}=\frac{1}{2p_K}$ for $j_1\neq j_2$. It can be verified that $\Sigma\Sigma^{-1}=\mathbf{I}_{K-1}$. 

If $p_j\neq q_j$ for some $1\leq j\leq K$, then $D_K>0$. Further, $D_{1,j}\neq 0$ and $D_{2,j}\neq 0$ iff $p_j\neq q_j$. Therefore, we can ignore the second order terms and get that as $n\to\infty$,
$$\sqrt{k_n}\suit{\hat D_K-D_K}\td \sum_{j=1}^K\suit{D_{1,j}N^X_j+D_{2,j}N^Y_j}\sim  N(0, \sigma^2).$$
The general result is the same as in Theorem \ref{thm_generic}. Under the setup of Theorem \ref{thm_AD_hom}, one can further calculate $\sigma^2$, which ends up with the result as shown therein.
\end{proof}

\section{Special case: the maximum risk functional}\label{sec:max_fct}

Finally, we show that a specific example with $d=2$ as in Example \ref{example: max_2} in which the conditions in Theorem \ref{thm_generic} hold, in particular, we can derive the limit distributions under both the null and alternative hypotheses. Denote $A(w,v):=\set{\g x: x^{(1)}>w, x^{(2)}>v}$ for all $w,v\in[0,+\infty)$. The example turns to be a specific division for the region $\Omega_r:=\set{x:r(\g x)>1}$ as 
\begin{align*}
  A_1&:=A(1,1)=\set{\g x: \min(x_1,x_2)>1},\\
  A_2&:=A(0,1)\setminus A(1,1)=\set{\g x: 0<x_1\leq 1, x_2>1},\\
  A_3&:=A(1,0)\setminus A(1,1)=\set{\g x: 0<x_2\leq 1, x_1>1}
\end{align*}

We show that the generic conditions in Theorem \ref{thm_generic} hold. We further show that the limit distribution under the null hypothesis is a scaled chi-squared distribution and derive the variance of the limit normal distribution under the alternative hypothesis. 

Similar to \eqref{cond_prob}, we define $p_j,q_j$ for $j=1,2,3$ corresponding to these sets and their $\hat p_j$ and $\hat q_j$, using the transformed observations. 
Recall that $\g{X}$ has standard Pareto distributed marginals. We apply the condition \eqref{mevd} to the sets $A(w,v)$ for any $w,v\geq 1$ and obtain that
\begin{align}\label{eq: definition of Lambda}
  \lim_{u\to \infty} u\mathbb P(\g{X}/u \in  A(w,v))&=\lim_{u\to \infty} \mathbb P(\g{X}/u \in  A(w,v)|\norm{\g X}_\infty>u) \cdot uP(\norm{\g X}_\infty>u) \nonumber\\
 &=\frac{\nu_X(A(w,v))}{\nu_X(A(0,1))}=:\Lambda\left(\frac{1}{w},\frac{1}{v}\right).
\end{align}
We remark that $\Lambda(1/w,1/v)>0$ for all $w,v\geq 1$. We can extend the definition of $\Lambda$ to the set $\bar{\mathbf{R}}_+^2=[0,+\infty]^2\setminus\set{(+\infty,+\infty)}$. We make the following assumption for the partial derivatives of $\Lambda$, see \cite{bucher2013}.

\textbf{Assumption (D)}:  Assume that the partial derivatives of $\Lambda$ exists, denoted as $\Lambda_1$ and $\Lambda_2$. In addition, $\Lambda_l$ is continuous on the set $\set{\g x: 0<x_l<+\infty}$ for $l=1,2$.
 
Last but not least, we assume the second-order condition \eqref{mevd_2nd} holds for $\g{X}$ and $\g{Y}$. Such an assumption is needed for proving the asymptotic behavior of the estimators for $p_j$ and $q_j$. The following corollary provides the asymptotic theory for the test statistics.

\begin{coro}\label{thm_AD_hetero}
 
  Assume that condition \eqref{mevd_2nd} holds for random vectors $\g{X}$ and $\g{ Y}$ with  $\nu_X$ and $\nu_Y$, and second order scale functions $a_{X}$ and $a_{Y}$ respectively. In addition, assume that Assumption (D) holds.
  
  Let $k_n$ be a sequence such that as $n\to\infty$, $k_n\to\infty$, $k_n/n \to 0$, $\sqrt{k_n} a_{X}(n/k_n)  \to 0$ and $\sqrt{k_n} a_{Y}(n/k_n)  \to 0$.
  
  Under the null hypothesis that $p_{j}=q_{j}$ for all $j=1,2,3$, the Kullback--Leibler divergence converges as follows: as $n \to \infty$,
$$
\frac{k_n}{2} \hat D_K  \td \suit{\sum_{j=1}^3\frac{1}{p_j}-1}\frac{(\Lambda-1)(\Lambda+2\Lambda_1\Lambda_2)}{\Lambda}\chi^2(1),
$$
where $\Lambda=\Lambda(\g 1)=\nu_X(\Omega_r)$, $\Lambda_j=\Lambda_j(\g 1)$ for $j=1,2$ are the partial derivatives of the function $\Lambda(\cdot)$ evaluated at the point $\g 1$.

Under the alternative hypothesis that $p_{j} \neq q_{j}$ for some $j=1,2,3$, denote by $D$ the population version in~\eqref{eq:KL_Multinomial} with $p_{j},q_{j}$. Then we have that as $n \to\infty$,
$$
\sqrt{k_n} \suit{\hat D_K -D_K} \td N(0, \sigma^2),
$$
where
\begin{align*}
  \sigma^2&=[D_{1,1}(1+p_1)+ D_{1,2}(-1+p_2)+D_{1,3}(-1+p_3)]^2\frac{(\Lambda^X-1)(\Lambda^X+2\Lambda^X_1\Lambda^X_2)}{\Lambda^X}\\
  &\:\:+[D_{1,1}(1+p_1)+ D_{1,2}(-1+p_2)+D_{1,3}(-1+p_3)]^2\frac{(\Lambda^Y-1)(\Lambda^Y+2\Lambda^Y_1\Lambda^Y_2)}{\Lambda^Y},
\end{align*}
with $D_{i,j}$ defined as in Theorem \ref{thm_AD_hom} for $i=1,2$ and $j=1,2,3$. Here the two $\Lambda$ functions based on $\tilde X$ and $\tilde Y$ are denoted as $\Lambda^X(\cdot)$ and $\Lambda^Y(\cdot)$ respectively. $\Lambda^X=\Lambda^X(\g 1)=\nu_X(\Omega_r)$, $\Lambda^X_j=\Lambda^X_j(\g 1)$ for $j=1,2$ are the partial derivatives of the function $\Lambda^X(\cdot)$ evaluated at the point $\g 1$. $\Lambda^Y$, $\Lambda^Y_j$ ($j=1,2$) are defined analogously based on the function $\Lambda^Y(\cdot)$.
\end{coro}

From the corollary, we observe that the marginal transformations can have an impact on the asymptotic behavior of the estimator for the Kullback--Leibler divergence, even under the null hypothesis. As a consequence, critical values derived from the $\chi^2$ distribution in Theorem \ref{thm_AD_hom} may not be valid. The actual asymptotic distribution may depend on the tail dependence structure. Nevertheless, since we can still derive the asymptotic distributions under both the null and alternative hypotheses with the same speed of convergence, we can use bootstrap method for deriving the critical values in the test.

To prove this Corollary, we first remark that $\Lambda$ defined in \eqref{eq: definition of Lambda} coincides with $\Lambda_L$ in \cite{bucher2013}, if we consider the random vector $\g{U_X}:=(1/X^{(1)},1/X^{(2)})$. Notice that $\g{U_X}$ follows a copula $C_X$ and
\begin{align*}
\lim_{t\to\infty}tC\left(\frac{1}{wt},\frac{1}{vt}\right)&=\lim_{t\to\infty}t\mathbb P\left(1/X^{(1)}\leq \frac{1}{wt},1/X^{(2)}\leq \frac{1}{vt}\right)\\
&=\lim_{t\to\infty}t\mathbb P\left(\g{X}\in tA(w,v)\right)=\Lambda\left(\frac{1}{w},\frac{1}{v}\right).
\end{align*}
Then we verify that all conditions requires in Theorem 2.2 in \cite{bucher2013} hold. Firstly, recall that the second order condition \eqref{mevd_2nd} holds for $\g{ X}$ with a second order scale function $a_{ X}(u)$. It implies that the limit relation in \eqref{eq: definition of Lambda} has the same speed of convergence determined by the function $a_{ X}(u)$. This is equivalent to (2.7) in \cite{bucher2013} with $A(t)=a_{ X}(t)$. Secondly, the intermediate sequence $k_n$ satisfies that $\sqrt{k_n}a_{ X}(n/k_n)\to 0$ as $n\to\infty$. Finally, our assumption (D) is the same as (2.14) therein. 

By applying Theorem 2.2 in \cite{bucher2013}, we immediately obtain the following Lemma, which is the main instrument used in proving Corollary \ref{thm_AD_hetero}.
\begin{lemma} \label{lem: from buecher and dette}
  Recall that $A(w,v)=\{\g x:x^{(1)}>w,x^{(2)}>v\}$. Define $$\hat\Lambda(x^{(1)},x^{(2)})=\frac{1}{k_n}\sum_{i=1}^n\einsfun\left\{\g{\hat X}_i\in \frac{n}{k_n}A\left(\frac{1}{x^{(1)}},\frac{1}{x^{(2)}}\right)\right\},$$
  for any $\g x=(x^{(1)},x^{(2)})\in \bar{\mathbf{R}}_+^2=[0,+\infty]^2\setminus\set{(+\infty,+\infty)}$. Then as $n\to\infty$,
  $$\sqrt{k_n}(\hat \Lambda(\g x)-\Lambda (\g x))\leadsto \mathbb G(\g x)-\Lambda_1(\g x)1 \mathbb G(x^{(1)},+\infty)-\Lambda_2(\g x)\mathbb G(+\infty, x^{(2)}),$$
  where $\Lambda_1$ and $\Lambda_2$ are the two partial derivatives of $\Lambda$, $\mathbb G$ is a Gaussian random fields with mean function zero and covariance structure $\mathbb E \mathbb G(\g x)\mathbb G(\g y)=\Lambda(\g x \wedge \g y)$. Here the convergence is in the space $\mathcal{B}_\infty (\bar{\mathbf{R}}_+^2)$ with well equipped norm. For details regarding the space and norm, see \cite{bucher2013}.
\end{lemma}

\begin{proof}[Proof of Corollary \ref{thm_AD_hetero}]
Define $I_1(t)=\hat\Lambda(1/t,1/t)$ for $t>0$.. We remark that its theoretical counterpart is $\Lambda(1/t,1/t)=\nu_X(A(t,t))=\frac{1}{t}\nu_X(A_1)$. We apply Lemma \ref{lem: from buecher and dette}  with $\g x$ replaced by $(1/t,1/t)$ and obtain the following result: as $n\to\infty$,
\begin{align}
  &\sqrt{k_n}\left(I_1(t)-\frac{\nu_X(A_1)}{t}\right)\nonumber\\
  \leadsto & 
      \:\:\mathbb G(1/t, 1/t)-\Lambda_1(1/t, 1/t)1 \mathbb G(1/t,+\infty)-\Lambda_2(1/t, 1/t)\mathbb G(+\infty, 1/t)\nonumber\\
      =: &\:\mathbb W(t),\label{eq: asymptotic for I_1}
\end{align}
where $\{\mathbb W(t)\}_{t\in(0,+\infty)}$ is a centered Gaussian process with a specific covariance structure. We only focus on the convergence on a compact set $t\in[a,b]$, where $0<a<b<+\infty$. Then the convergence is uniform on $[a,b]$.

Next consider $I_2(t)=\hat\Lambda(1/t,+\infty)$ and $I_3(u)=\hat\Lambda(+\infty,1/t)$. In fact, these two quantities are non-random since marginals of $\g{\hat X}$ are deterministic values $n/1, n/2, \cdots, n/n$. It is straightforward to verify that as $n\to\infty$,
$$\sqrt{k_n}(I_j(t)-1/t)\to 0,$$
for all $t\in [a,b]$ and $j=2,3$. The result can also be obtained by applying Lemma \ref{lem: from buecher and dette}.

Define 
$$I(t)=I_2(t)+I_3(t)-I_1(t)=\frac{1}{k_n}\sum_{i=1}^n\einsfun\left\{r(\g{\hat X}_i)>\frac{n}{k_n}t\right\},$$
where $r(\g x)=\max(\g x)$. Then we get that as $n\to\infty$,
$$\sqrt{k_n}\left(I(t)-\frac{1}{t}\nu_X(\Omega_r))\right)\leadsto -\mathbb W(t),$$
for $t\in [a,b]$.

Under a proper Skorokhod construction, we can have the convergence above in the almost surely sense. Choose particularly $a=\nu_X(\Omega_r)/2$ and $b=2\nu_X(\Omega_r)$. Then by applying the Vervaat's lemma to this limit relation, we get that as $n\to\infty$, for $v\in [1/2,2]$,
$$
\sqrt{k_n}\suit{\frac{k_n}{n}R^{\hat X}_{n-[k_nv],n}-v^{-1}\nu_X(\Omega_r)}\leadsto -\frac{\nu_X(\Omega_r)}{v^2}\mathbb W\suit{v^{-1}\nu_X(\Omega_r)}.
$$
By setting $v=1$, we get that  as $n\to\infty$,
\begin{equation} \label{eq: asymptotic for threshold}
\sqrt{k_n}\suit{\frac{k_n}{n}R^{\hat X}_{n-k_n,n}-\nu_X(\Omega_r)}\leadsto -\nu_X(\Omega_r)\mathbb W\suit{\nu_X(\Omega_r)}.
\end{equation}
Therefore, we can replace $t=\frac{k_n}{n}R^{\hat X}_{n-k_n,n}$ in \eqref{eq: asymptotic for I_1}. Notice that such a replacement leads to exactly the estimator $\hat p_1$ since
$$I_1\suit{\frac{k_n}{n}R^{\hat X}_{n-[k_n],n}}=\frac{1}{k_n}\sum_{i=1}^n\einsfun\left\{\g{\hat X}_i\in R^{\hat X}_{n-[k_n],n} A_1\right\}=\hat p_1$$
Hence we get that as $n\to\infty$
$$\sqrt{k_n}\suit{\hat p_1-\suit{\frac{k_n}{n}R^{\hat X}_{n-[k_n],n}}^{-1}\nu(A_1)}\leadsto \mathbb W(\nu_X(\Omega_r)).$$
Together with \eqref{eq: asymptotic for threshold} and the fact that $p_1=\nu(A_1)/\nu(\Omega_r)$, we get that as $n\to\infty$,
$$\sqrt{k_n}\suit{\hat p_1-p_1}\leadsto (1+p_1)\mathbb W(\nu_X(\Omega_r)).$$
Similarly, by considering $I_2(t)-I_1(t)$ and replace $t=\frac{k_n}{n}R^{\hat X}_{n-k_n,n}$, we can obtain that as $n\to\infty$
$$\sqrt{k_n}\suit{\hat p_2-\suit{\frac{k_n}{n}R^{\hat X}_{n-[k_n],n}}^{-1}\nu(A_2)}\leadsto -\mathbb W(\nu_X(\Omega_r)).$$
Together with \eqref{eq: asymptotic for threshold} and the fact that $p_2=\nu(A_2)/\nu(\Omega_r)$, we get that as $n\to\infty$,
$$\sqrt{k_n}\suit{\hat p_2-p_2}\leadsto (-1+p_2)\mathbb W(\nu_X(\Omega_r)).$$
Similarly, we have that as $n\to\infty$,
$$\sqrt{k_n}\suit{\hat p_3-p_3}\leadsto (-1+p_3)\mathbb W(\nu_X(\Omega_r)).$$
Denote $N_X=\mathbb W(\nu_X(\Omega_r))$. We get that a joint convergence for the estimators: as $n\to\infty$,
\begin{align*}
  \sqrt{k_n}\left(\begin{matrix}
    \hat p_1-p_1\\
    \hat p_2-p_2\\
    \hat p_3-p_3
  \end{matrix}\right)\td \left(\begin{matrix}
    1+p_1\\
    -1+p_2\\
    -1+p_3
  \end{matrix}\right)N_X.
\end{align*}
In fact a straightforward calculation gives the variance of $N_X=\mathbb W(\nu_X(\Omega_r))$ as 
$$\mathbb V(N_X)=\frac{(\Lambda-1)(\Lambda+2\Lambda_1\Lambda_2)}{\Lambda},$$
where $\Lambda=\Lambda(\g 1)=\nu_X(\Omega_r)$, $\Lambda_j=\Lambda_j(\g 1)$ for $j=1,2$.

We can derive the asymptotic behavior for $\hat q_j$ $j=1,2,3$ in a similar way with a limit random variable $N_Y$ independent of $N_X$.

Under the null hypothesis that $p_j=q_j$ for $j=1,2,3$, similar to the proof of Theorem \ref{thm_AD_hom}, we get that as $n\to\infty$,
\begin{align*}
\frac{k_n}{2}\hat D_K&\td \frac{1}{2p_1}(1+p_1)^2(N_X-N_Y)^2+\frac{1}{2p_2}(-1+p_2)^2(N_X-N_Y)^2+\frac{1}{2p_3}(-1+p_3)^2(N_X-N_Y)^2\\
&\stackrel{d}{=}\suit{\sum_{j=1}^3\frac{1}{p_j}-1}\frac{(\Lambda-1)(\Lambda+2\Lambda_1\Lambda_2)}{\Lambda}\chi^2(1)
\end{align*}

Under the alternative hypothesis, similar to the proof of Theorem \ref{thm_AD_hom}, we get that as $n\to\infty$,
\begin{align*}
  &\sqrt{k_n}\suit{\hat D_K-D_K}\\
  \td &  [D_{1,1}(1+p_1)+ D_{1,2}(-1+p_2)+D_{1,3}(-1+p_3)]N_X\\
  &+[D_{2,1}(1+q_1)+ D_{2,2}(-1+q_2)+D_{2,3}(-1+2_3)]N_Y,
\end{align*}
which results in the asymptotic normality. Further calculations leads to the variance of the limit distribution as in the Corollary.
\end{proof}

\section{Divergence as function of extremal correlation}\label{extremal_correlation}

Consider the maximum risk functional with the same three sets $A_1, A_2, A_3$ as in Appendix~\ref{sec:max_fct}, and with corresponding 
probabilities $p_j, q_j$, $j=1,2,3$. Note that by standardization, the probabilities 
$p_2 = p_3$ and $q_2=q_3$. Moreover, note that $1 = p_1+p_2+p_3 = p_1 + 2p_2$, and therefore $p_2 = (1-p_1)/2$.
The extremal correlation coefficient corresponding to the
exponent measure $\nu_X$ is given by $\chi_X = p_1 / (p_1+p_2) = 2 p_1 / (1 + p_1)$. Solving this we obtain $p_1 = \chi_X / (2-\chi_X)$, and similarly $q_1 = \chi_Y / (2-\chi_Y)$. We the also have $p_2 = p_3 = (1-\chi_X)/(2-\chi_X)$, and similarly for $q_2=q_3$. Plugging this into the definition of the KL divergence in~\eqref{eq:KL_Multinomial} we get
\begin{align*}
    D_3 &= \left(\frac{\chi_X}{2 - \chi_X} - \frac{\chi_Y}{2 - \chi_Y} \right)\left(\frac{\chi_X}{2 - \chi_X} - \log \frac{\chi_Y}{2 - \chi_Y} \right)\\
    &+ 2\left(\frac{1-\chi_X}{2 - \chi_X} - \frac{1-\chi_Y}{2 - \chi_Y} \right)\left(\frac{1-\chi_X}{2 - \chi_X} - \log \frac{1-\chi_Y}{2 - \chi_Y} \right).
\end{align*}
This shows that the KL divergence for the maximum risk functional is 
a simple function of the extremal correlation coefficients.

\section{Additional simulations for power of test}\label{app:simu}

Similar to Figure~\ref{fig:clay_clay_H0}, Figure~\ref{fig:clay_clay_H0_max} shows the 
divergence and rejection percentages in 
the case where the risk functionals is the
maximum risk functional $r(\mathbf x) = \max(x_1,x_2)$.

\begin{figure}[tb]
	\centering
	\hspace*{-2.7em}\includegraphics[width=0.38\linewidth]{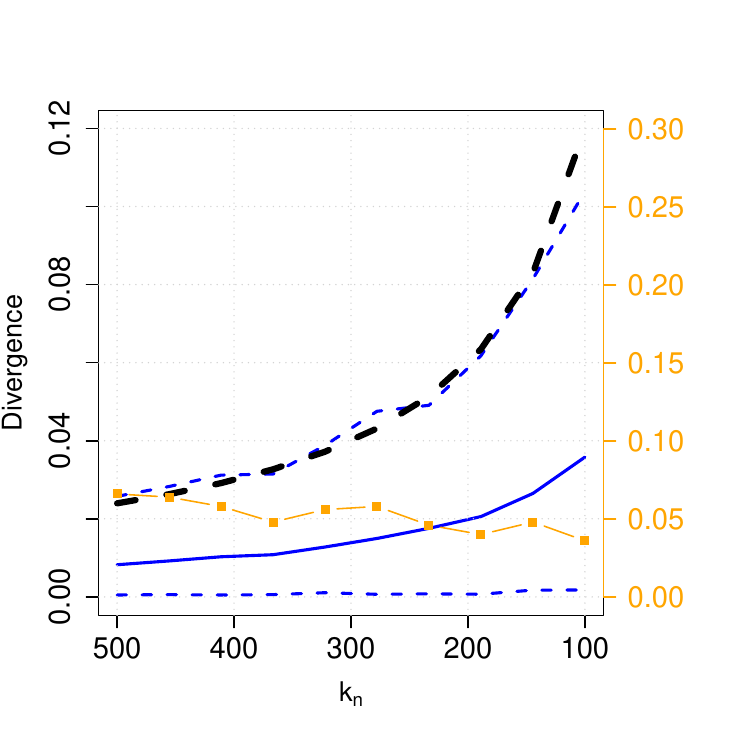}
  \hspace*{-1.8em}\includegraphics[width=0.38\linewidth]{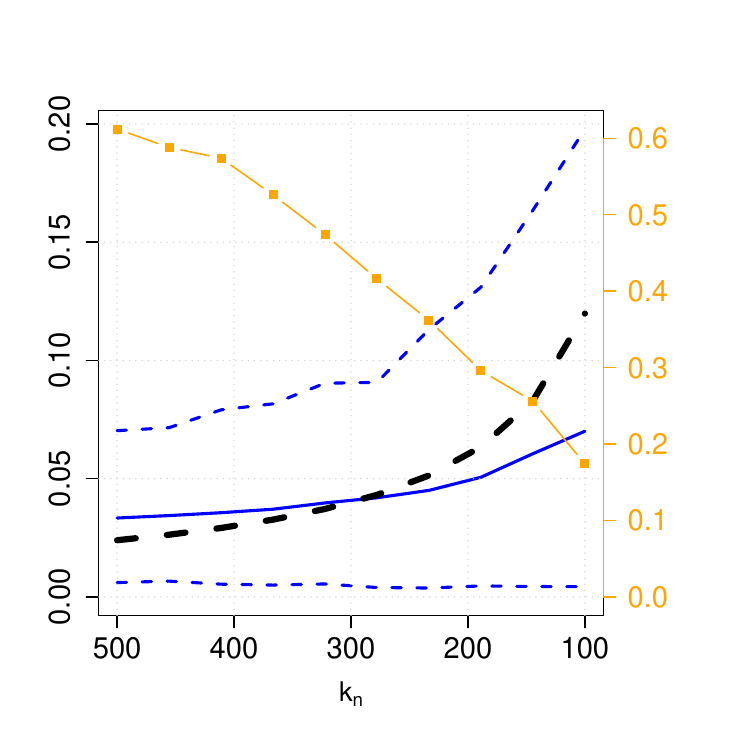}
	\hspace*{-1.8em}\includegraphics[width=0.38\linewidth]{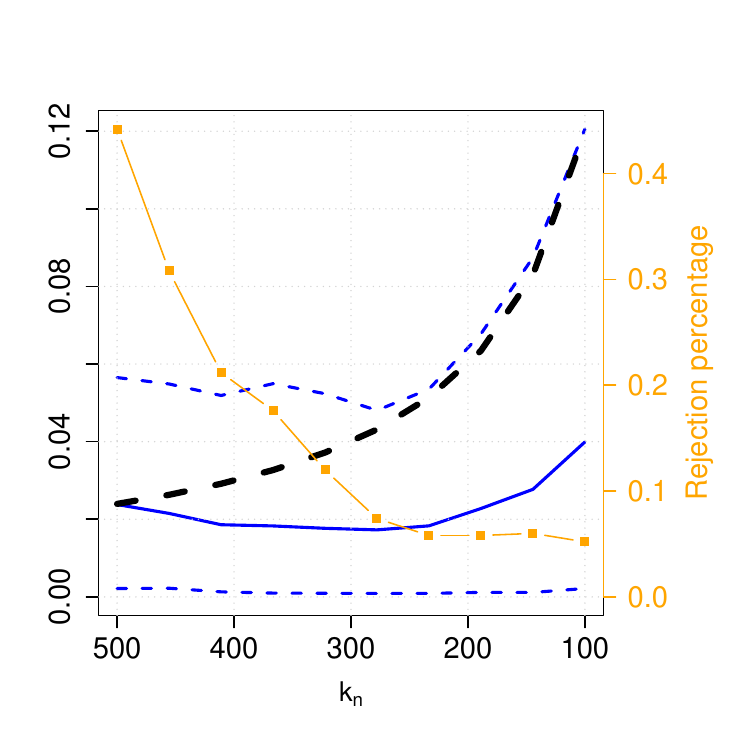}
 \hspace*{-3em}\\ \vspace*{-4.5em}%
	\hspace*{-2.7em}\includegraphics[width=0.38\linewidth]{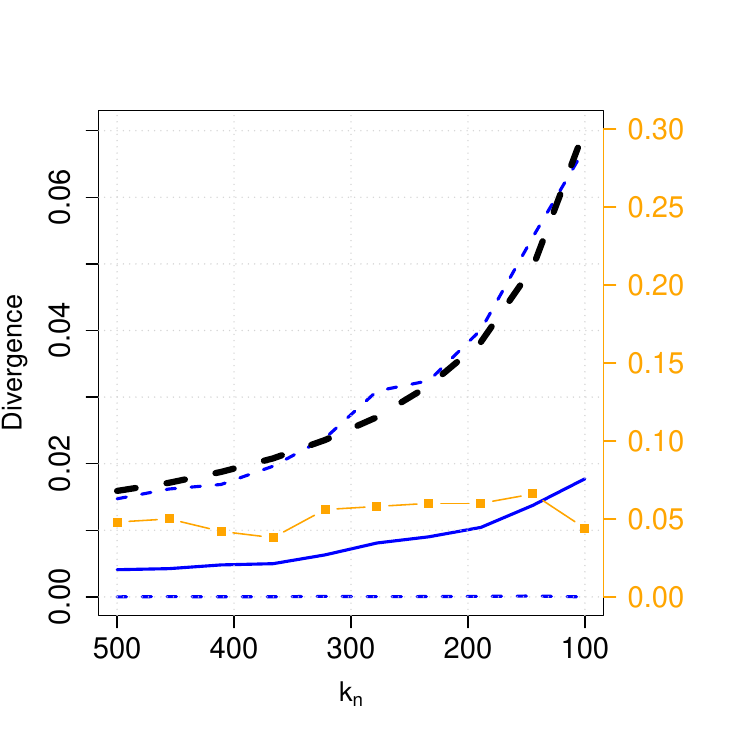}
  \hspace*{-1.8em}\includegraphics[width=0.38\linewidth]{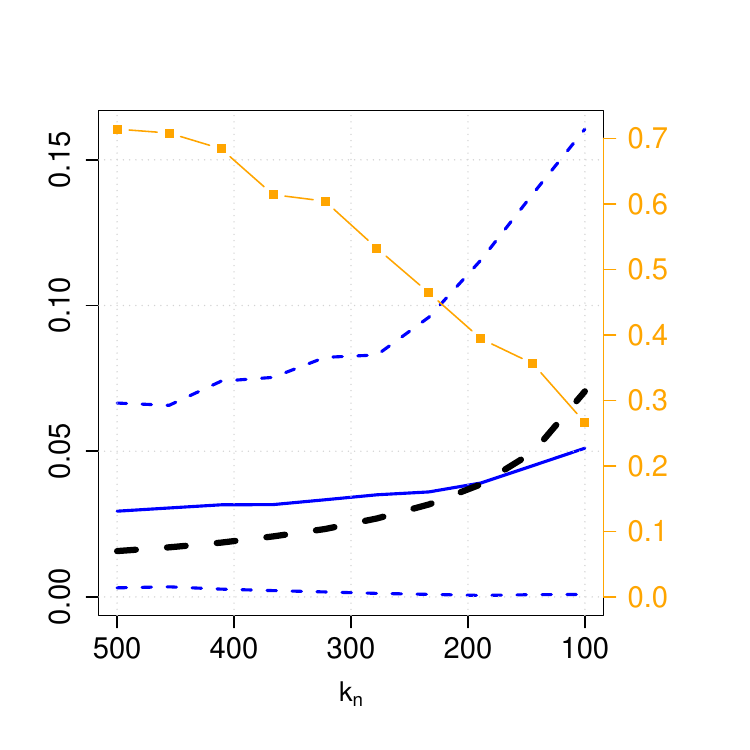}
	\hspace*{-1.8em}\includegraphics[width=0.38\linewidth]{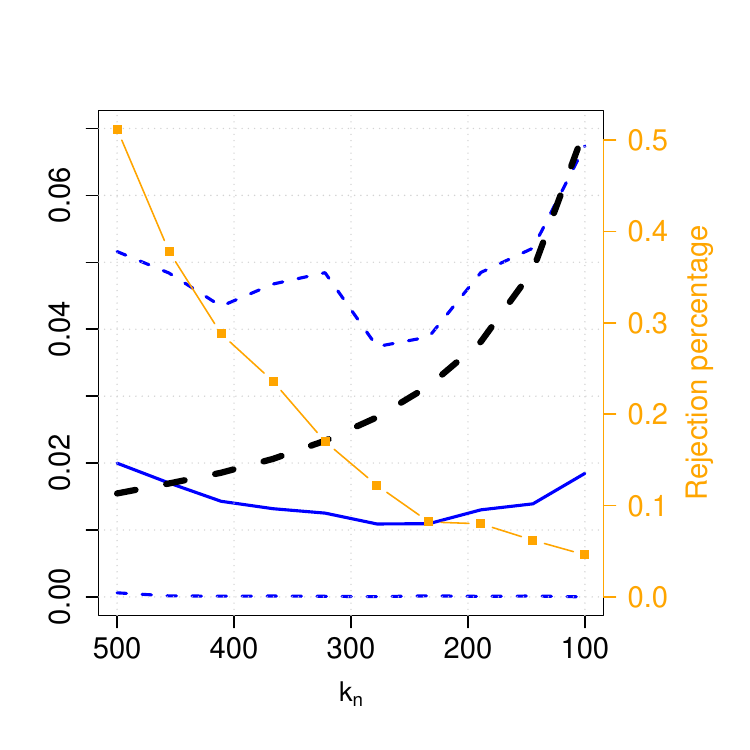}
 \vspace*{-2em}
  \caption{Mean (solid blue line) and empirical $5\%$ and $95\%$ quantiles (dashed blue lines) of 500 samples of KL test statistic $\hat D_K$ based on maximum risk functional with $K=3$ sets for a range of exceedances $k_n$
  together with the critical values (black dashed line) at level $95\%$ and rejection percentages (orange line). Top and bottom rows show results for known and unknown margins, respectively. 
   The two samples are generated from the same distribution (left), from distributions with different extremal dependence structure (center), and from different distributions with same extremal dependence structure (right).}
 \label{fig:clay_clay_H0_max}
\end{figure}

\end{document}